\documentclass[fleqn, pdftex]{article} 

\usepackage{amsmath,amssymb,amsthm}
\usepackage{tikz}

\title{Modal completeness of sublogics of the interpretability logic $\mathbf{IL}$}
\author{Taishi Kurahashi and Yuya Okawa}
\date{}

\theoremstyle{plain}
\newtheorem{thm}{Theorem}[section]
\newtheorem{lem}[thm]{Lemma}
\newtheorem{prop}[thm]{Proposition}
\newtheorem{cor}[thm]{Corollary}

\newtheorem{prob}[thm]{Problem}

\theoremstyle{definition}
\newtheorem{defn}[thm]{Definition}

\newtheorem{rem}[thm]{Remark}

\newcommand{\PA}{\mathbf{PA}}

\newcommand{\GL}{\mathbf{GL}}
\newcommand{\CL}{\mathbf{CL}}
\newcommand{\IL}{\mathbf{IL}}
\newcommand{\ILS}{\mathbf{IL}^-_{\mathrm{set}}}
\newcommand{\ILM}{\mathbf{ILM}}
\newcommand{\ILP}{\mathbf{ILP}}

\newcommand{\rank}{\mathrm{rank}}

\newcommand{\G}[1]{\mathbf{L#1}}
\newcommand{\J}[1]{\mathbf{J#1}}
\newcommand{\R}[1]{\mathbf{R#1}}

\begin{document}

\maketitle

\begin{abstract}
We study modal completeness and incompleteness of several sublogics of the interpretability logic $\IL$. 
We introduce the sublogic $\IL^-$, and prove that $\IL^-$ is sound and complete with respect to Veltman prestructures which are introduced by Visser. 
Moreover, we prove the modal completeness of twelve logics between $\IL^-$ and $\IL$ with respect to Veltman prestructures.  
On the other hand, we prove that eight natural sublogics of $\IL$ are modally incomplete. 
Finally, we prove that these incomplete logics are complete with respect to generalized Veltman prestructures. 
As a consequence of these investigations, we obtain that the twenty logics studied in this paper are all decidable. 
\end{abstract}

\section{Introduction}
The notion of formalized provability is well-studied in the framework of modal logic. 
The provability logic of Peano Arithmetic $\PA$ is the set of all modal formulas that are verifiable in $\PA$ when the modal operator $\Box$ is interpreted as the provability predicate $\mathrm{Pr}_{\PA}(x)$ of $\PA$. 
Solovay's arithmetical completeness theorem \cite{Sol76} states that the provability logic of $\PA$ is exactly axiomatized by the modal logic $\GL$ that is obtained from the smallest normal modal logic $\mathbf{K}$ by adding the axiom scheme $\Box(\Box A \to A) \to \Box A$. 
Segerberg \cite{Seg71} proved that the logic $\GL$ is sound and complete with respect to the class of all transitive and conversely well-founded finite Kripke frames. 

The interpretability logic $\IL$ is the base logic for modal logical investigations of the notion of relative interpretability. 
The language of $\IL$ is that of $\GL$ with the additional binary modal operator $\rhd$. 
The binary modal operator $\rhd$ binds stronger than $\to$, but weaker than $\lnot$, $\land$, $\lor$ and $\Box$. 
The intended meaning of the formula $A \rhd B$ is ``$\PA + B$ is relatively interpretable in $\PA + A$''. 
The inference rules of $\IL$ are the same as those of $\GL$, and the axioms of $\IL$ are those of $\GL$ together with the following axioms: 

\begin{description}
	\item [$\J{1}$] $\Box (A \to B) \to A \rhd B$; 
	\item [$\J{2}$] $(A \rhd B) \land (B \rhd C) \to A \rhd C$; 
	\item [$\J{3}$] $(A \rhd C) \land (B \rhd C) \to (A \lor B) \rhd C$; 
	\item [$\J{4}$] $A \rhd B \to (\Diamond A \to \Diamond B)$; 
	\item [$\J{5}$] $\Diamond A \rhd A$. 
\end{description}

The logic $\IL$ is not arithmetically complete by itself. 
The logic $\ILM$ is obtained from $\IL$ by adding Montagna's Principle $A \rhd B \to (A \land \Box C) \rhd (B \land \Box C)$, then Berarducci \cite{Ber90} and Shavrukov \cite{Sha88} independently proved that $\ILM$ is arithmetically sound and complete with respect to arithmetical interpretations for $\PA$. 
Also let $\ILP$ be the logic $\IL$ with the Persistence Principle $A \rhd B \to \Box (A \rhd B)$. 
Visser \cite{Vis90} proved the arithmetical completeness theorem of the logic $\ILP$ with respect to arithmetical interpretations for suitable finitely axiomatized fragments of $\PA$. 

These logics have Kripkean semantics. 
A triple $\langle W, R, \{S_x\}_{x \in W} \rangle$ is said to be an \textit{$\IL$-frame} or a \textit{Veltman frame} if $\langle W, R \rangle$ is a Kripke frame of $\GL$ and for each $x \in W$, $S_x$ is a transitive and reflexive binary relation on $R[x] : = \{y \in W : x R y\}$ satisfying the following property: $(\ast)$ $\forall y, z \in W(x R y\ \&\ y R z \Rightarrow y S_x z)$. 
De Jongh and Veltman \cite{deJVel90} proved that $\IL$ is sound and complete with respect to all finite $\IL$-frames. 
Also they proved that the logics $\ILM$ and $\ILP$ are sound and complete with respect to corresponding classes of finite $\IL$-frames, respectively. 

The logic $\IL$ and its extensions are not only arithmetically significant. 
It is known that for extensions of $\PA$, relative interpretability is equivalent to $\Pi_1$-conservativity, and this equivalence is provable in $\PA$ (see \cite{JapdeJ98}). 
Therefore the logic $\ILM$ is also the logic of $\Pi_1$-conservativity for $\PA$ (see also \cite{HajMon90}). 
On the other hand, when we consider the logics of $\Gamma$-conservativity for $\Gamma \neq \Pi_1$, the principle $\J{5}$ is no longer arithmetically valid. 
Ignatiev \cite{Ign91} introduced the logic of conservativity $\CL$ which is obtained from $\IL$ by removing $\J{5}$, and he proved that the extensions $\mathbf{SbCLM}$ and $\mathbf{SCL}$ of $\CL$ are exactly the logic of $\Pi_2$-conservativity and the logic of $\Gamma$-conservativity for $\Gamma \in \{\Sigma_n, \Pi_n : n \geq 3\}$, respectively. 

Ignatiev also proved that $\CL$ is complete with respect to Kripkean semantics. 
A triple $\langle W, R, \{S_x\}_{x \in W} \rangle$ is said to be a \textit{$\CL$-frame} if it is a structure with all properties of $\IL$-frame but $(\ast)$. 
Then $\CL$ is sound and complete with respect to the class of all finite $\CL$-frames. 
The correspondence between $\J{5}$ and the property $(\ast)$ is explained in the framework of $\IL^-$-frames. 
A triple $\langle W, R, \{S_x\}_{x \in W} \rangle$ is called an \textit{$\IL^-$-frame} or a \textit{Veltman prestructure} if $\langle W, R \rangle$ is a frame of $\GL$ and for each $x \in W$, $S_x$ is a binary relation on $W$ with $\forall y, z \in W(y S_x z \Rightarrow x R y)$. 
Then Visser \cite{Vis88} stated that for any $\IL^-$-frame, the validity of the scheme $\J{5}$ is equivalent to the property $(\ast)$. 

Visser also showed, for example, that for any $\IL^-$-frame, the validity of $\J{4}$ is equivalent to the property $\forall x, y, z \in W(y S_x z \Rightarrow x R z)$. 
However, systematic study of sublogics of $\IL$ through $\IL^-$-frames has not been done so far. 
In this paper, we do this study, and prove the modal completeness and incompleteness of several sublogics of $\IL$. 

In Section \ref{Sec:IL^-}, we introduce the logic $\IL^-$ that is valid in all $\IL^-$-frames. 
We introduce the notion of $\ILS$-frames that serves a wider class of models than $\IL^-$-frames. 
Then we show that $\IL^-$ is also valid in all $\ILS$-frames. 
In Section \ref{Sec:Ext}, we investigate several axiom schemata and extensions of $\IL^-$. 
Section \ref{Sec:Lem} is devoted to proving lemmas used to prove our modal completeness theorems.
Our modal completeness theorem with respect to $\IL^-$-frames is proved in Section \ref{Sec:Compl}. 
In Section \ref{Sec:Incompl}, we prove several natural sublogics of $\IL$ are incomplete with respect to $\IL^-$-frames. 
In Section \ref{Sec:GCompl}, we prove these incomplete logics are complete with respect to $\ILS$-frames. 
Finally, Section \ref{Sec:CR} concludes the present paper with a few words.

\section{The logic $\IL^-$}\label{Sec:IL^-}

In this section, we introduce and investigate the logic $\IL^-$. 
The language of $\IL^-$ consists of countably many propositional variables $p, q, r, \ldots$, logical constants $\top$, $\bot$, connectives $\neg$, $\land$, $\lor$, $\to$ and modal operators $\Box$, $\rhd$. 
The expression $\Diamond A$ is an abbreviation for $\lnot \Box \lnot A$. 
We show that every theorem of $\IL^-$ is valid in all $\IL^-$-frames defined below (see Definition \ref{Def:ILf}). 
In fact, we will prove in Section \ref{Sec:Compl} that $\IL^-$ is sound and complete with respect to the class of all (finite) $\IL^-$-frames. 
The logic $\IL^-$ is the basis for our logics discussed in this paper. 

First, we introduce the logic $\IL^-$. 
Note that the axioms and rules of the logic $\IL^-$ are intended to characterize the logic of the class of all $\IL^{-}$-frames.

\begin{defn}
The axiom schemata of the logic $\IL^-$ are as follows: 
\begin{description}
	\item [$\G{1}$] All tautologies in the language of $\IL^-$; 
	\item [$\G{2}$] $\Box(A \to B) \to (\Box A \to \Box B)$; 
	\item [$\G{3}$] $\Box(\Box A \to A) \to \Box A$; 
	\item [$\J{3}$] $(A \rhd C) \land (B \rhd C) \to (A \lor B) \rhd C$; 
	\item [$\J{6}$] $\Box A \leftrightarrow (\neg A) \rhd \bot$. 
\end{description}
The inference rules of $\IL^-$ are Modus Ponens $\dfrac{A\ \ \ A \to B}{B}$, Necessitation $\dfrac{A}{\Box A}$, $\R{1}$ and $\R{2}$. 
Here the rules $\R{1}$ and $\R{2}$ are defined as follows: 
\begin{description}
	\item [$\R{1}$] $\dfrac{A \to B}{C \rhd A \to C \rhd B}$; 
	\item [$\R{2}$] $\dfrac{A \to B}{B \rhd C \to A \rhd C}$. 
\end{description}
\end{defn}

The logic $\GL$ consists of the axiom schemata $\G{1}, \G{2}$ and $\G{3}$, and of the inference rules Modus Ponens and Necessitation (in the language without $\rhd$). 
Hence $\IL^-$ is an extension of $\GL$. 
In Subsection \ref{SSec:CLIL}, we prove that $\IL$ proves the axiom $\J{6}$ and admits the rules $\R{1}$ and $\R{2}$ (see Proposition \ref{CLP1}).
Therefore $\IL^-$ is a sublogic of $\IL$.

We introduce $\IL^-$-frames that are originally introduced by Visser \cite{Vis88} as Veltman prestructures. 

\begin{defn}\label{Def:ILf}
We say that a triple $\langle W, R, \{S_x\}_{x \in W} \rangle$ is an \textit{$\IL^-$-frame} if it satisfies the following conditions: 
\begin{enumerate}
	\item $W$ is a non-empty set; 
	\item $R$ is a transitive and conversely well-founded binary relation on $W$; 
	\item For each $x \in W$, $S_x$ is a binary relation on $W$ satisfying $\forall y, z \in W(y S_x z \Rightarrow x R y)$. 
\end{enumerate}
A quadruple $\langle W, R, \{S_x\}_{x \in W}, \Vdash \rangle$ is called an \textit{$\IL^-$-model} if $\langle W, R, \{S_x\}_{x \in W} \rangle$ is an $\IL^-$-frame and $\Vdash$ is a binary relation between $W$ and the set of all formulas satisfying the usual conditions for satisfaction with the following conditions: 
\begin{itemize}
	\item $x \Vdash \Box A \iff \forall y \in W(x R y \Rightarrow y \Vdash A)$. 
	\item $x \Vdash A \rhd B \iff \forall y \in W(x R y \ \&\ y \Vdash A \Rightarrow \exists z \in W (y S_x z\ \&\ z \Vdash B))$. 
\end{itemize}
A formula $A$ is said to be \textit{valid} in an $\IL^-$-frame $\langle W, R, \{S_x\}_{x \in W} \rangle$ if for all satisfaction relations $\Vdash$ on the frame and all $x \in W$, $x \Vdash A$. 
\end{defn}

For each $x \in W$, let $R[x] : = \{y \in W : x R y\}$. 
In this notation, the third clause in the definition of $\IL^-$-frames states that $S_x$ is a relation on $R[x] \times W$. 
Note that this clause can be removed from the definition because it is not forced by axiom schemata of $\IL^-$ and does not affect the definition of $\Vdash$. 
We impose this clause to simplify our arguments.

We prove that $\IL^-$ is sound with respect to the class of all $\IL^-$-frames. 

\begin{prop}\label{0P1}
Every theorem of $\IL^-$ is valid in all $\IL^-$-frames. 
\end{prop}
\begin{proof}
We prove the claim by induction on the length of proofs in $\IL^-$. 
Since the modal logic $\GL$ is sound with respect to the class of all transitive and conversely well-founded Kripke frames (see \cite{Boo93}), all theorems of $\GL$ in the language of $\IL^{-}$ are valid in all $\IL^{-}$-frames. 
That is, $\G{1}$, $\G{2}$ and $\G{3}$ are valid in all $\IL^{-}$-frames, and the rules Modus Ponens and Necessitation preserve the validity. 
Then it suffices to prove that $\J{3}$ and $\J{6}$ are valid in all $\IL^-$-frames, and the rules $\R{1}$ and $\R{2}$ preserve the validity. 

Let $F = \langle W, R, \{S_x\}_{x \in W} \rangle$ be an $\IL^{-}$-frame, $x \in W$ be any element and $\Vdash$ be any satisfaction relation on $F$. 

$\J{3}$: Suppose $x \Vdash (A \rhd C) \land (B \rhd C)$. 
Let $y \in W$ be any element with $x R y$ and $y \Vdash A \lor B$. 
In either case of $y \Vdash A$ and $y \Vdash B$, there exists $z \in W$ such that $y S_x z$ and $z \Vdash C$. 
Thus we obtain $x \Vdash (A \lor B) \rhd C$. 

$\J{6}$: $(\rightarrow)$: 
Suppose $x \Vdash \Box A$. 
Then there is no $y \in W$ such that $x R y$ and $y \Vdash \neg A$. 
Hence $x \Vdash (\neg A) \rhd \bot$. 

$(\leftarrow)$: 
Suppose $x \Vdash (\neg A) \rhd \bot$. 
If there were $y \in W$ with $x R y$ and $y \Vdash \neg A$, then there would be some $z \in W$ such that $z \Vdash \bot$, a contradiction. 
Thus if $x R y$, then $y \Vdash A$, and this means $x \Vdash \Box A$. 

$\R{1}$: Assume $A \to B$ is valid in $F$. 
Suppose $x \Vdash C \rhd A$ and let $y \in W$ be such that $x R y$ and $y \Vdash C$. 
Then there exists $z \in W$ such that $y S_x z$ and $z \Vdash A$. 
By the assumption, $z \Vdash B$. 
Then we obtain $x \Vdash C \rhd B$. 

$\R{2}$: Assume $A \to B$ is valid in $F$. 
Suppose $x \Vdash B \rhd C$ and let $y \in W$ be such that $x R y$ and $y \Vdash A$. 
By the assumption, $y \Vdash B$, and hence there exists $z \in W$ such that $y S_x z$ and $z \Vdash C$. 
Thus we have $x \Vdash A \rhd C$. 
\end{proof}

By the rules $\R{1}$ and $\R{2}$, we immediately obtain the following proposition. 

\begin{prop}\label{0P2}
Let $L$ be a logic with the inference rules $\R{1}$ and $\R{2}$. 
If $L \vdash A_0 \leftrightarrow A_1$ and $L \vdash B_0 \leftrightarrow B_1$, then $L \vdash A_0 \rhd B_0 \leftrightarrow A_1 \rhd B_1$. 
\end{prop}

In this paper, we freely use Proposition \ref{0P2} without any mention. 
In $\IL^-$, the inference rule $\R{2}$ is strengthened as follows. 

\begin{prop}\label{0P3}\leavevmode
\begin{enumerate}
	\item $\IL^- \vdash \Box \neg A \to A \rhd B$. 
	\item $\IL^- \vdash \Box(A \to B) \to (B \rhd C \to A \rhd C)$. 
\end{enumerate}
\end{prop}
\begin{proof}
1. Since $\IL^- \vdash \bot \to B$, we have $\IL^- \vdash A \rhd \bot \to A \rhd B$ by the rule $\R{1}$. 
By the axiom $\J{6}$, we obtain $\IL^- \vdash \Box \neg A \to A \rhd B$. 

2. Since $\IL^- \vdash \Box(A \to B) \to \Box \neg(A \land \neg B)$, we have $\IL^- \vdash \Box(A \to B) \to (A \land \neg B) \rhd C$ by 1. 
Then $\IL^- \vdash \Box(A \to B) \land (B \rhd C) \to ((A \land \neg B)\lor B) \rhd C$ by the axiom $\J{3}$. 
Since $\IL^- \vdash A \to (A \land \neg B) \lor B$, we have $\IL^- \vdash ((A \land \neg B) \lor B) \rhd C \to A \rhd C$ by the rule $\R{2}$. 
Therefore we conclude $\IL^- \vdash \Box(A \to B) \land (B \rhd C) \to A \rhd C$. 
\end{proof}

Thus $\IL^-$ is deductively equivalent to the system obtained from $\IL^-$ by replacing the rule $\R{2}$ by the axiom scheme $\Box(A \to B) \to (B \rhd C \to A \rhd C)$. 

In Section \ref{Sec:Compl}, we will prove that several extensions of $\IL^-$ are complete with respect to corresponding classes of $\IL^-$-frames. 
On the other hand, we will also prove that several logics are not complete. 
To prove this incompleteness, we use the notion of $\ILS$-frames that is a general notion of $\IL_{\mathrm{set}}$-frames or generalized Veltman frames introduced by Verbrugge \cite{Ver92} (see also \cite{Vuk96,Joo98}). 

\begin{defn}\label{Def:ILS}
A tuple $\langle W, R, \{S_x\}_{x \in W} \rangle$ is called an \textit{$\ILS$-frame} if it satisfies the following conditions: 
\begin{enumerate}
	\item $W$ is a non-empty set; 
	\item $R$ is a transitive and conversely well-founded binary relation on $W$; 
	\item For each $x \in W$, $S_x$ is a relation on $W \times (\mathcal{P}(W) \setminus \{\emptyset\})$ such that $\forall y \in W, \forall V \subseteq W (y S_x V \Rightarrow x R y)$; 
	\item (Monotonicity) $\forall x, y \in W, \forall V, U \subseteq W (y S_x V\ \&\ V \subseteq U \Rightarrow y S_x U)$. 
\end{enumerate}
As in the definition of $\IL^-$-frames, we can define $\ILS$-models $\langle W, R, \{S_x\}_{x \in W}, \Vdash \rangle$ with the following clause: 
\begin{itemize}
	\item $x \Vdash A \rhd B \iff \forall y \in W(x R y\ \&\ y \Vdash A \Rightarrow \exists V \subseteq W (y S_x V\ \&\ \forall z \in V (z \Vdash B)))$. 
\end{itemize}
\end{defn}

Let $M = \langle W, R, \{S_x\}_{x \in W}, \Vdash \rangle$ be an $\IL^-$-model. 
For each $x \in W$, we define the relation $S_x' \subseteq W \times (\mathcal{P}(W) \setminus \{\emptyset\})$ by $y S_x' V : \iff \exists z \in V (y S_x z)$. 
Then it is shown that $\langle W, R, \{S_x'\}_{x \in W} \rangle$ is an $\ILS$-frame. 
Let $\Vdash'$ be the unique satisfaction relation on this $\ILS$-frame satisfying that for any $x \in W$ and any propositional variable $p$, $x \Vdash' p$ if and only if $x \Vdash p$. 
Then $\langle W, R, \{S_x'\}_{x \in W}, \Vdash' \rangle$ is an $\ILS$-model, and for any $x \in W$ and any formula $A$, $x \Vdash A$ if and only if $x \Vdash' A$.  
Therefore, in this sense, every $\IL^-$-frame (resp.~model) can be recognized as an $\ILS$-frame (resp.~model). 
We strengthen Proposition \ref{0P1}. 

\begin{prop}\label{0P4}
Every theorem of $\IL^-$ is valid in all $\ILS$-frames. 
\end{prop}
\begin{proof}
Let $F = \langle W, R, \{S_x\}_{x \in W} \rangle$ be an $\ILS$-frame, $x \in W$ be any element and $\Vdash$ be any satisfaction relation on $F$. 
As in the proof of Proposition \ref{0P1}, we only prove the cases $\J{3}$, $\J{6}$, $\R{1}$ and $\R{2}$. 

	$\J{3}$: 
	Suppose $x \Vdash (A \rhd C) \land (B \rhd C)$. 
	Let $y \in W$ be any element such that $x R y$ and $y \Vdash A \lor B$. 
	In either case that $y \Vdash A$ and $y \Vdash B$, there exists $V \subseteq W$ such that $y S_x V$ and $\forall z \in V (z \Vdash C)$. 
	Therefore $x \Vdash(A \lor B) \rhd C$. 
	
	$\J{6}$: This follows from the following equivalences: 
\begin{align*}
	x \Vdash \Box A & \iff \forall y(x R y \Rightarrow y \Vdash A), \\
	& \iff \forall y(x R y\ \&\ y \Vdash \neg A \Rightarrow \exists V(V \neq \emptyset \ \&\  y S_x V \ \&\ \forall z \in V(z \Vdash \bot))),\\
	& \iff x \Vdash (\neg A) \rhd \bot. 
\end{align*}

	$\mathbf{R1}$: 
	Assume that $A \to B$ is valid in $F$. 
	Suppose $x \Vdash C \rhd A$. 
	Let $y \in W$ be such that $x R y$ and $y \Vdash C$. 
	Then there exists $V \subseteq W$ such that $y S_x V$ and $\forall z \in V(z \Vdash A)$. 
	For each $z \in V$, $z \Vdash B$ by the assumption. 
	Therefore we conclude $x \Vdash C \rhd B$.

	$\mathbf{R2}$: 
	Assume that $A \to B$ is valid in $F$. 
	Suppose $x \Vdash B \rhd C$ and let $y \in W$ be any element with $x R y$ and $y \Vdash A$. 
	Then $y \Vdash B$ by the assumption, and hence there exists $V \subseteq W$ such that $y S_x V$ and $\forall z \in V(z \Vdash C)$. 
	Thus we have $x \Vdash A \rhd C$. 

\end{proof}

\begin{rem}
In the proof of Proposition \ref{0P4}, Monotonicity of $\ILS$-frames is not used at all. 
As in the case of $\IL^{-}$-frames, conditions 3 and 4 in Definition \ref{Def:ILS} are introduced because they are useful properties to have. 

In usual definition of $\IL_{\mathrm{set}}$-frame $\langle W, R, \{S_x\}_{x \in W} \rangle$,  each $S_x$ is a relation on $R[x] \times  (\mathcal{P}(R[x]) \setminus \{\emptyset\})$. 
Therefore $\IL_{\mathrm{set}}$-frames are not $\ILS$-frames because of our definition of Monotonicity. 
On the other hand, from Proposition \ref{4P6} below, when we deal with logics containing $\J{4}_{+}$,  we can restrict conditions 3 and 4 in Definition \ref{Def:ILS} so that $S_{x}$ is a relation on $R[x] \times  (\mathcal{P}(R[x]) \setminus \{\emptyset\})$. 
\end{rem}

\section{Extensions of $\IL^-$}\label{Sec:Ext}

In this section, we investigate several additional axiom schemata and several extensions of $\IL^-$. 
Let $\Sigma_1, \ldots, \Sigma_n$ be axiom schemata. 
Then $\IL^-(\Sigma_1, \ldots, \Sigma_n)$ is the logic $\IL^-$ together with the axiom schemata $\Sigma_1, \ldots, \Sigma_n$. 
Let $L$ be an extension of $\IL^-$. 
We say that $L$ is \textit{complete with respect to finite $\IL^-$-frames} (resp.~$\ILS$-frames) if for any formula $A$, $L \vdash A$ if and only if $A$ is valid in all finite $\IL^-$-frames (resp.~$\ILS$-frames) where all axioms of $L$ are valid. 

\subsection{The axiom scheme $\J{1}$}

In this subsection, we investigate the axiom scheme $\J{1}$. 
\begin{description}
	\item [$\J{1}$] $\Box (A \to B) \to A \rhd B$.  
\end{description}
First, we show that the following axiom scheme $\J{1}'$ is equivalent to $\J{1}$ over $\IL^-$. 
\begin{description}
	\item [$\J{1}'$] $A \rhd A$.  
\end{description}

\begin{prop}\label{1P1}
The logics $\IL^-(\J{1})$ and $\IL^-(\J{1}')$ are deductively equivalent. 
\end{prop}
\begin{proof}
$\IL^-(\J{1}) \vdash \J{1}'$: 
This is because $\IL^- \vdash \Box(A \to A)$ and $\IL^-(\J{1}) \vdash \Box(A \to A) \to A \rhd A$. 

$\IL^-(\J{1}') \vdash \J{1}$: 
By Proposition \ref{0P3}.2, $\IL^- \vdash \Box(A \to B) \to (B \rhd B \to A \rhd B)$. 
Since $\IL^-(\J{1}') \vdash B \rhd B$, we obtain the desired result. 
\end{proof}

Therefore, in this paper, we sometimes identify the axiom schemata $\J{1}$ and $\J{1}'$. 
The following proposition is stated in Visser. 

\begin{prop}[Visser \cite{Vis88}]\label{1P2}
Let $F = \langle W, R, \{S_x\}_{x \in W} \rangle$ be any $\IL^-$-frame. 
Then the following are equivalent: 
\begin{enumerate}
	\item $\J{1}$ is valid in $F$. 
	\item $\forall x, y \in W(x R y \Rightarrow y S_x y)$. 
\end{enumerate}
\end{prop}
\begin{proof}
$(1 \Rightarrow 2)$: 
Assume that $\J{1}$ is valid in $F$. 
Suppose $x R y$. 
Let $\Vdash$ be any satisfaction relation on $F$ satisfying that for any $u \in W$, $u \Vdash p$ if and only if $u = y$ for some fixed propositional variable $p$. 
Then $y \Vdash p$. 
Since $x R y$ and $x \Vdash p \rhd p$, there exists a $z \in W$ such that $y S_x z$ and $z \Vdash p$. 
By the definition of $\Vdash$, $z = y$, and hence $y S_x y$. 

$(2 \Rightarrow 1)$: 
Assume $\forall x, y \in W(x R y \Rightarrow y S_x y)$. 
Let $y \in W$ be such that $x R y$ and $y \Vdash A$. 
Then $y S_x y$, and thus we conclude $x \Vdash A \rhd A$. 
\end{proof}

We prove a similar equivalence concerning $\ILS$-frames. 

\begin{prop}\label{1P3}
Let $F = \langle W, R, \{S_x\}_{x \in W} \rangle$ be any $\ILS$-frame. 
Then the following are equivalent: 
\begin{enumerate}
	\item $\J{1}$ is valid in $F$.  
	\item $\forall x, y \in W(x R y \Rightarrow y S_x \{y\})$. 
\end{enumerate}
\end{prop}
\begin{proof}
$(1 \Rightarrow 2)$: 
Assume that $\J{1}$ is valid in $F$. 
Suppose $x R y$. 
Let $\Vdash$ be a satisfaction relation on $F$ satisfying for any $u \in W$, $u \Vdash p$ if and only if $u = y$ for some fixed propositional variable $p$. 
Then $y \Vdash p$. 
Since $x R y$ and $x \Vdash p \rhd p$, there exists $V \subseteq W$ such that $y S_x V$ and $\forall z \in V (z \Vdash p)$. 
By the definition of $\Vdash$, $V = \{y\}$ because $V$ is non-empty. 
We obtain $y S_x \{y\}$. 

$(2 \Rightarrow 1)$: 
Assume $\forall x, y \in W(x R y \Rightarrow y S_x \{y\})$. 
Let $y \in W$ be such that $x R y$ and $y \Vdash A$. 
Then $y S_x \{y\}$ and $\forall z \in \{y\} (z \Vdash A)$. 
Thus we conclude $x \Vdash A \rhd A$. 
\end{proof}

\subsection{The axiom scheme $\J{4}$}

This subsection is devoted to investigating the axiom scheme $\J{4}$. 
\begin{description}
	\item [$\J{4}$] $(A \rhd B) \to (\Diamond A \to \Diamond B)$. 
\end{description}

First, we prove that $\J{4}$ is equivalent to the following axiom scheme $\J{4}'$ over $\IL^-$. 
The principle $\J{4}'$ is introduced in Visser \cite{Vis88}. 

\begin{description}
	\item [$\J{4}'$] $(A \rhd B) \to (B \rhd \bot \to A \rhd \bot)$. 
\end{description}

\begin{prop}\label{4P1}
The logics $\IL^-(\J{4})$ and $\IL^-(\J{4}')$ are deductively equivalent. 
\end{prop}
\begin{proof}
This is because $\IL^- \vdash (\Diamond A \to \Diamond B) \leftrightarrow (B \rhd \bot \to A \rhd \bot)$ by $\J{6}$. 
\end{proof}

Since $\J{4}'$ is a particular instance of the axiom scheme $\J{2}$, we obtain the following corollary. 

\begin{cor}\label{4C1}
$\IL^-(\J{2}) \vdash \J{4}$. 
\end{cor}

The axiom scheme $\J{4}$ does not behave well by itself in the sense of modal completeness. 
In fact, we will prove in Section \ref{Sec:Incompl} that for instance, $\IL^-(\J{4})$ is not complete with respect to corresponding class of $\IL^-$-frames. 
Thus we introduce a well-behaved axiom scheme $\J{4}_{+}$ whose corresponding class of $\IL^-$-frames is same as that of $\J{4}$. 
The principle $\J{4}_+$ is originally introduced in Visser \cite{Vis88}. 
We also introduce the schemata $\J{4}_+'$ and $\J{4}_+''$ as follows: 

\begin{description}
	\item [$\J{4}_+$] $\Box(A \to B) \to (C \rhd A \to C \rhd B)$. 
	\item [$\J{4}_+'$] $\Box A \to (C \rhd (A \to B) \to C \rhd B)$. 
	\item [$\J{4}_+''$] $\Box A \to (C \rhd B \to C \rhd (A \land B))$. 
\end{description}

\begin{prop}\label{4P2}
The logics $\IL^-(\J{4}_+)$, $\IL^-(\J{4}_+')$ and $\IL^-(\J{4}_+'')$ are deductively equivalent. 
\end{prop}
\begin{proof}
$\IL^-(\J{4}_+) \vdash \J{4}_+'$: 
Since $\IL^- \vdash A \to ((A \to B) \to B)$, $\IL^- \vdash \Box A \to \Box((A \to B) \to B)$. 
Then we have $\IL^-(\J{4}_+) \vdash \Box A \to (C \rhd (A \to B) \to C \rhd B)$. 

$\IL^-(\J{4}_+') \vdash \J{4}_+''$: 
Since $\IL^- \vdash B \to (A \to A \land B)$, $\IL^- \vdash C \rhd B \to C \rhd (A \to A \land B)$ by the rule $\R{1}$. 
Then $\IL^-(\J{4}_+') \vdash \Box A \to (C \rhd B \to C \rhd (A \land B))$. 

$\IL^-(\J{4}_+'') \vdash \J{4}_+$: 
By the axiom $\J{4}_+''$, we have $\IL^-(\J{4}_+'') \vdash \Box (A \to B) \to (C \rhd A \to C \rhd ((A \to B) \land A))$. 
Since $\IL^- \vdash (A \to B) \land A \to B$, we have $\IL^- \vdash C \rhd ((A \to B) \land A) \to C \rhd B$ by the rule $\R{1}$. 
Thus $\IL^-(\J{4}_+'') \vdash \Box(A \to B) \to (C \rhd A \to C \rhd B)$. 
\end{proof}

The axiom scheme $\J{4}_+$ is a strengthening of the inference rule $\R{1}$, and hence in extensions of $\IL^-(\J{4}_+)$, the inference rule $\R{1}$ is redundant. 

We show that $\J{4}_+$ implies $\J{4}$ over $\IL^-$. 

\begin{prop}\label{4P3}
$\IL^-(\J{4}_+) \vdash \J{4}$. 
\end{prop}
\begin{proof}
Since $\IL^- \vdash B \rhd \bot \to \Box \neg B$ by $\J{6}$, $\IL^- \vdash B \rhd \bot \to \Box (B \to \bot)$. 
Then by $\J{4}_+$, we have $\IL^-(\J{4}_+) \vdash A \rhd B \to (B \rhd \bot \to A \rhd \bot)$. 
By Proposition \ref{4P1}, we conclude $\IL^-(\J{4}_+) \vdash \J{4}$. 
\end{proof}

We prove that $\J{4}$ and $\J{4}_+$ have the same frame condition with respect to $\IL^-$-frames. 
This is stated in Visser \cite{Vis88}.  

\begin{prop}[Visser \cite{Vis88}]\label{4P4}
Let $F = \langle W, R, \{S_x\}_{x \in W} \rangle$ be any $\IL^-$-frame. 
Then the following are equivalent:
\begin{enumerate}
	\item $\J{4}_+$ is valid in $F$.  
	\item $\J{4}$ is valid in $F$.  
	\item $\forall x, y, z \in W(y S_x z \Rightarrow x R z)$.  
\end{enumerate}
\end{prop}

\begin{proof}
$(1 \Rightarrow 2)$: By Proposition \ref{4P3}. 

$(2 \Rightarrow 3)$: Assume that $\mathbf{J4}$ is valid in $F$. Suppose $y S_{x} z$. 
Let $\Vdash$ be a satisfaction relation on $F$ such that for any $u \in W$, $u \Vdash p$ if and only if $u=y$, and $u \Vdash q$ if and only if $u = z$ for some fixed propositional variables $p$ and $q$. 
Then $x \Vdash p \rhd q$ by the definition of $\Vdash$ and our supposition. 
Since $x R y$ and $y \Vdash p$, we have $x \Vdash \Diamond p$. 
Then by the validity of $\J{4}$, $x \Vdash \Diamond q$. 
Hence there exists $u \in W$ such that $x R u$ and $u \Vdash q$. 
By the definition of $\Vdash$, we obtain $x R z$.

$(3 \Rightarrow 1)$: Assume that $\forall x, y, z \in W(y S_x z \Rightarrow x R z)$. 
Suppose $x \Vdash A \rhd B$ and $x \Vdash \Diamond A$. 
Then there exists $y \in W$ such that $xRy$ and $y \Vdash A$ and hence there exists $z \in W$ such that $yS_{x}z$ and $z \Vdash B$. 
By the assumption $x R z$ and therefore we obtain $x \Vdash \Diamond B$. 
That is, $\J{4_{+}}$ is valid in $F$. 
\end{proof}

On the other hand, $\J{4}$ and $\J{4}_+$ can be distinguished by considering $\ILS$-frames. 
That is, these logics have different frame conditions with respect to $\ILS$-frames. 

\begin{prop}\label{4P5}
Let $F = \langle W, R, \{S_x\}_{x \in W} \rangle$ be any $\ILS$-frame. 
Then the following are equivalent: 
\begin{enumerate}
	\item $\J{4}$ is valid in $F$.  
	\item $\forall x, y \in W, \forall V \subseteq W (y S_x V \Rightarrow \exists z \in V (x R z))$. 
\end{enumerate}
\end{prop}
\begin{proof}
$(1 \Rightarrow 2)$: 
Assume that $\J{4}$ is valid in $F$, and suppose $y S_x V$. 
Let $\Vdash$ be a satisfaction relation on $F$ such that for any $u \in W$, $u \Vdash p$ if and only if $u = y$, and $u \Vdash q$ if and only if $u \in V$ for some fixed propositional variables $p$ and $q$. 
Then $x \Vdash p \rhd q$ because $V$ is non-empty. 
Since $x R y$ and $y \Vdash p$, we have $x \Vdash \Diamond p$. 
Then by the validity of $\J{4}$, $x \Vdash \Diamond q$. 
Hence there exists $z \in W$ such that $x R z$ and $z \Vdash q$. 
By the definition of $\Vdash$, we obtain $z \in V$. 

$(2 \Rightarrow 1)$:
Assume $\forall x, y \in W, \forall V \subseteq W(y S_x V \Rightarrow \exists z \in V (x R z))$. 
Suppose $x \Vdash (A \rhd B) \land \Diamond A$. 
Then there exists $y \in W$ such that $x R y$ and $y \Vdash A$, and also there exists a $V \subseteq W$ such that $y S_x V$ and $\forall z \in V (z \Vdash B)$. 
By the assumption, $x R z$ for some $z \in V$. 
Hence $x \Vdash \Diamond B$. 
This shows that $\J{4}$ is valid in $F$. 
\end{proof}

\begin{prop}\label{4P6}
Let $F = \langle W, R, \{S_x\}_{x \in W} \rangle$ be any $\ILS$-frame. 
Then the following are equivalent: 
\begin{enumerate}
	\item $\J{4}_+$ is valid in $F$.  
	\item $\forall x, y \in W, \forall V \subseteq W(y S_x V \Rightarrow y S_x (V \cap R[x]))$. 
\end{enumerate}
\end{prop}
\begin{proof}
$(1 \Rightarrow 2)$: 
Assume that $\J{4}_+$ is valid in $F$. 
Suppose $y S_x V$. 
Let $\Vdash$ be a satisfaction relation on $F$ such that for any $u \in W$, $u \Vdash p$ if and only if $u = y$, $u \Vdash q$ if and only if $u \in V$, and $u \Vdash r$ if and only if ($u \in V$ and $x R u$), for some fixed propositional variables $p, q$ and $r$. 
Then $x \Vdash p \rhd q$ because $V$ is non-empty. 
Let $y \in W$ be any element such that $x R y$ and $y \Vdash q$, then $y \in V$ and $x R y$. 
This means $y \Vdash r$. 
Therefore $x \Vdash \Box (q \to r)$. 
By the validity of $\J{4}_+$, we obtain $x \Vdash p \rhd r$. 
Since $x R y$ and $y \Vdash p$, there exists a $U \subseteq W$ such that $y S_x U$ and $\forall z \in U(z \Vdash r)$. 
By the definition of $\Vdash$, for each $z \in U$, $z \in V$ and $x R z$. 
That is, $U \subseteq V \cap R[x]$. 
By Monotonicity, we conclude $y S_x (V \cap R[x])$. 

$(2 \Rightarrow 1)$: Assume $\forall x, y \in W, \forall V \subseteq W(y S_x V \Rightarrow y S_x (V \cap R[x]))$. 
Suppose $x \Vdash (A \rhd B) \land \Box(B \to C)$. 
Let $y \in W$ be such that $x R y$ and $y \Vdash A$, then there exists a $V \subseteq W$ such that $y S_x V$ and $\forall z \in V (z \Vdash B)$. 
By the assumption, $y S_x (V \cap R[x])$. 
In particular, for each $z \in V \cap R[x]$, $z \Vdash B$ and $z \Vdash B \to C$, and hence $z \Vdash C$. 
We have shown $x \Vdash A \rhd C$. 
Therefore $\J{4}_+$ is valid in $F$. 
\end{proof}

By Proposition \ref{4P6}, when we consider logics containing $\J{4}_+$, for each $\ILS$-frame $\langle W, R, \{S_x\}_{x \in W} \rangle$, we may assume that for every $x \in W$, $S_x$ is a relation on $R[x] \times (\mathcal{P}(R[x]) \setminus \{\emptyset\})$. 
This is required in the usual definition of $\IL_{\textrm{set}}$-frames (see \cite{Vuk96,Joo98}).  

\subsection{The axiom scheme $\J{2}$}

In this subsection, we discuss the axiom scheme $\J{2}$.
\begin{description}
	\item [$\J{2}$] $(A \rhd B) \land (B \rhd C) \to A \rhd C$. 
\end{description}

As in the case of the axiom $\J{4}$, we introduce the following new axiom schemata $\J{2}_+$ and $\J{2}_+'$ which are stronger than $\J{2}$. 

\begin{description}
	\item [$\J{2}_+$] $(A \rhd (B \lor C)) \land (B \rhd C) \to A \rhd C$. 
	\item [$\J{2}_+'$] $(A \rhd B) \land ((B \land \neg C) \rhd C) \to A \rhd C$. 
\end{description}

\begin{prop}\label{2P1}
The logics $\IL^-(\J{2}_+)$ and $\IL^-(\J{2}_+')$ are deductively equivalent. 
\end{prop}
\begin{proof}
$\IL^-(\J{2}_+) \vdash \J{2}_+'$: 
Since $\IL^- \vdash B \to (B \land \neg C) \lor C$, we have $\IL^- \vdash A \rhd B \to A \rhd ((B \land \neg C) \lor C)$ by the rule $\R{1}$. 
Then $\IL^- \vdash (A \rhd B) \land ((B \land \neg C) \rhd C) \to (A \rhd ((B \land \neg C) \lor C)) \land ((B \land \neg C) \rhd C)$. 
Thus $\IL^-(\J{2}_+) \vdash (A \rhd B) \land ((B \land \neg C) \rhd C) \to A \rhd C$. 

$\IL^-(\J{2}_+') \vdash \J{2}_+$: 
Since $\IL^- \vdash (B \lor C) \land \neg C \to B$, $\IL^- \vdash B \rhd C \to ((B \lor C) \land \neg C) \rhd C$ by the rule $\R{2}$. 
Then $\IL^- \vdash (A \rhd (B \lor C)) \land (B \rhd C) \to (A \rhd (B \lor C)) \land ((B \lor C) \land \neg C) \rhd C$. 
Therefore we conclude $\IL^-(\J{2}_+') \vdash (A \rhd (B \lor C)) \land (B \rhd C) \to A \rhd C$. 
\end{proof}

The axiom scheme $\J{2}_+$ is slightly stronger than $\J{2}$. 
In fact, the following proposition shows that $\J{2}$ and $\J{2}_+$ are equivalent over the logic $\IL^-(\J{1})$. 

\begin{prop}\label{2P2}\leavevmode
\begin{enumerate}
	\item $\IL^-(\J{2}_+) \vdash \J{2}$. 
	\item $\IL^-(\J{1}, \J{2}) \vdash \J{2}_+$. 
\end{enumerate}
\end{prop}
\begin{proof}
1. This is because $\IL^- \vdash A \rhd B \to A \rhd (B \lor C)$. 

2. Since $\IL^-(\J{1}) \vdash B \rhd C \to (B \lor C) \rhd C$ by $\J{1}$ and $\J{3}$, we have $\IL^-(\J{1}, \J{2}) \vdash (A \rhd (B \lor C)) \land (B \rhd C) \to A \rhd C$. 
\end{proof}

We proved in Corollary \ref{4C1} that $\IL^-(\J{2})$ proves $\J{4}$. 
Analogously, we prove that $\J{2}_+$ is stronger than $\J{4}_+$ over $\IL^-$. 

\begin{prop}\label{2P3}
$\IL^-(\J{2}_+) \vdash \J{4}_+$. 
\end{prop}
\begin{proof}
Since $\IL^- \vdash \Box(A \to B) \to \Box \neg (A \land \neg B)$, we have $\IL^- \vdash \Box(A \to B) \to (A \land \neg B) \rhd B$ by Proposition \ref{0P3}.1. 
Then $\IL^-(\J{2}_+') \vdash \Box(A \to B) \to (C \rhd A \to C \rhd B)$. 
By Proposition \ref{2P1}, we obtain $\IL^-(\J{2}_+) \vdash \J{4}_+$. 
\end{proof}

The following corollary is straightforward from Propositions \ref{2P2}.2 and \ref{2P3}. 

\begin{cor}\label{2C1}
$\IL^-(\J{1}, \J{2}) \vdash \J{4}_+$. 
\end{cor}

We prove that $\J{2}$ and $\J{2}_+$ have the same frame condition with respect to the $\IL^-$-frames. 

\begin{prop}\label{2P4}
Let $F = \langle W, R, \{S_x\}_{x \in W} \rangle$ be any $\IL^-$-frame. 
Then the following are equivalent: 
\begin{enumerate}
	\item $\J{2}_+$ is valid in $F$. 
	\item $\J{2}$ is valid in $F$. 
	\item $\J{4}$ is valid in $F$ and for any $x \in W$, $S_x$ is transitive. 
\end{enumerate}
\end{prop}
\begin{proof}
$(1 \Rightarrow 2)$: By Proposition \ref{2P2}.1. 

$(2 \Rightarrow 3)$: This is proved in Visser \cite{Vis88}. 

$(3 \Rightarrow 1)$: Assume that $\mathbf{J4}$ is valid in $F$ and for any $x \in W$, $S_x$ is transitive. 
Suppose $x \Vdash (A \rhd (B \lor C)) \land (B \rhd C)$. 
Let $y \in W$ be any element such that $x R y$ and $y \Vdash A$. 
Then there exists $z \in W$ such that $y S_x z$ and $z \Vdash B \lor C$. 
We shall show that there exists $u \in W$ such that $y S_x u$ and $u \Vdash C$. 
If $z \Vdash C$, then this is done. 
If $z \nVdash C$, then $z \Vdash B$. 
Since $x R z$, by our supposition, there exists $u \in W$ such that $z S_x u$ and $u \Vdash C$. 
By the transitivity of $S_x$, we obtain $y S_x u$. 

Therefore we conclude $x \Vdash A \rhd C$. 
That is to say, $\J{2}_+$ is valid in $F$. 
\end{proof}

We prove that $\J{2}$ and $\J{2}_+$ have different frame conditions with respect to $\ILS$-frames. 

\begin{prop}\label{2P5}
Let $F = \langle W, R, \{S_x\}_{x \in W} \rangle$ be an $\ILS$-frame. 
Then the following are equivalent: 
\begin{enumerate}
	\item $\J{2}$ is valid in $F$.  
	\item $\J{4}$ is valid in $F$ and 
\[
	\forall x, y \in W, \forall V \subseteq W \left(y S_x V\ \&\ \forall z \in V \cap R[x] (z S_x U_z) \Rightarrow y S_x \left(\bigcup_{z \in V \cap R[x]} U_z\right)\right).
\] 
\end{enumerate}
\end{prop}
\begin{proof}
$(1 \Rightarrow 2)$: 
Assume that $\J{2}$ is valid in $F$. 
Then by Corollary \ref{4C1}, $\J{4}$ is valid in $F$. 
Suppose $y S_x V$ and $\forall z \in V \cap R[x] (z S_x U_z)$. 
Let $\Vdash$ be a satisfaction relation on $F$ such that for any $u \in W$, $u \Vdash p$ if and only if $u = y$, $u \Vdash q$ if and only if $u \in V$, and $u \Vdash r$ if and only if $\exists z \in V \cap R[x] (u \in U_z)$. 
Then $x \Vdash p \rhd q$ and $x \Vdash q \rhd r$. 
By the validity of $\J{2}$, $x \Vdash p \rhd r$. 
Since $x R y$ and $y \Vdash p$, there exists a $U \subseteq W$ such that $y S_x U$ and $\forall w \in U (w \Vdash r)$. 
By the definition of $\Vdash$, $U \subseteq \bigcup_{z \in V \cap R[x]} U_z$. 
By Monotonicity, $y S_x (\bigcup_{z \in V \cap R[x]} U_z)$.

$(2 \Rightarrow 1)$: 
Assume that $\J{4}$ is valid in $F$ and $\forall x, y \in W, \forall V \subseteq W(y S_x V\ \&\ \forall z \in V \cap R[x](z S_x U_z) \Rightarrow y S_x (\bigcup_{z \in V \cap R[x]} U_z))$. 
Suppose $x \Vdash (A \rhd B) \land (B \rhd C)$. 
Let $y \in W$ be any element with $x R y$ and $y \Vdash A$. 
Then there exists a $V \subseteq W$ such that $y S_x V$ and $\forall z \in V (z \Vdash B)$. 
Since $\J{4}$ is valid in $F$, we have $V \cap R[x] \neq \emptyset$. 
Then for each $z \in V \cap R[x]$,  there exists $U_z \subseteq W$ such that $z S_x U_z$ and $\forall w \in U_z (w \Vdash C)$. 
By the assumption, $y S_x (\bigcup_{z \in V \cap R[x]} U_z)$ because the set $\bigcup_{z \in V \cap R[x]} U_z$ is non-empty. 
Also $\forall w \in \bigcup_{z \in V \cap R[x]} U_z (w \Vdash C)$. 
We have shown $w \Vdash A \rhd C$. 
Hence $\J{2}$ is valid in $F$. 
\end{proof}

The condition $\forall x, y \in W, \forall V \subseteq W (y S_x V\ \&\ \forall z \in V \cap R[x] (z S_x U_z) \Rightarrow y S_x (\bigcup_{z \in V \cap R[x]} U_z))$ stated in Proposition \ref{2P5} is required in the usual definition of $\IL_{\textrm{set}}$-frames. 

\begin{prop}\label{2P6}
Let $F = \langle W, R, \{S_x\}_{x \in W} \rangle$ be any $\ILS$-frame. 
Then the following are equivalent: 
\begin{enumerate}
	\item $\J{2}_+$ is valid in $F$.  
	\item $\J{4}$ is valid in $F$ and
\[
	\forall x, y \in W, \forall V_0, V_1 \subseteq W \left(y S_x (V_0 \cup V_1)\ \&\ \forall z \in V_0 \cap R[x] (z S_x U_z) \Rightarrow y S_x \left(\bigcup_{z \in V_0 \cap R[x]} U_z \cup V_1 \right) \right).
\] 
\end{enumerate}
\end{prop}
\begin{proof}
$(1 \Rightarrow 2)$: 
Assume $\J{2}_+$ is valid in $F$. 
Since $\IL^-(\J{2}_+) \vdash \J{4}$, $\J{4}$ is also valid in $F$. 
Suppose $y S_x (V_0 \cup V_1)$ and $\forall z \in V_0 \cap R[x] (z S_x U_z)$. 
Let $\Vdash$ be a satisfaction relation on $F$ such that for any $u \in W$, $u \Vdash p$ if and only if $u = y$, $u \Vdash q$ if and only if $u \in V_0$, and $u \Vdash r$ if and only if ($\exists z \in V_0 \cap R[x] (u \in U_z)$ or $u \in V_1$). 
Then $x \Vdash p \rhd (q \lor r)$ and $x \Vdash q \rhd r$. 
By the validity of $\J{2}_+$, $x \Vdash p \rhd r$. 
Since $x R y$ and $y \Vdash p$, there exists a $U \subseteq W$ such that $y S_x U$ and $\forall w \in U (w \Vdash r)$. 
Then by the definition of $\Vdash$, we have $U \subseteq \bigcup_{z \in V_0 \cap R[x]} U_z \cup V_1$. 
By Monotonicity, $y S_x (\bigcup_{z \in V_0 \cap R[x]} U_z \cup V_1)$.

$(2 \Rightarrow 1)$: 
Assume that $\J{4}$ is valid in $F$ and $\forall x, y \in W, \forall V_0, V_1 \subseteq W(y S_x (V_0 \cup V_1)\ \&\ \forall z \in V_0 \cap R[x](z S_x U_z) \Rightarrow y S_x (\bigcup_{z \in V_0 \cap R[x]} U_z \cup V_1))$. 
Let $x \Vdash (A \rhd (B \lor C)) \land (B \rhd C)$. 
Let $y \in W$ be such that $x R y$ and $y \Vdash A$, then there exists a $V \subseteq W$ such that $y S_x V$ and $\forall z \in V (z \Vdash B \lor C)$. 
Since $\J{4}$ is valid, we have $V \cap R[x] \neq \emptyset$. 
Let $V_0 : = \{z \in V : z \Vdash B\}$ and $V_1 : = \{z \in V : z \Vdash C\}$, then $V = V_0 \cup V_1$. 
In particular, for each $z \in V_0 \cap R[x]$, there exists a $U_z \subseteq W$ such that $z S_x U_z$ and $\forall w \in U_z (w \Vdash C)$. 
By the assumption, we have $y S_x (\bigcup_{z \in V_0 \cap R[x]} U_z \cup V_1)$ because $\bigcup_{z \in V_0 \cap R[x]} U_z \cup V_1$ is non-empty. 
Since $\forall w \in \bigcup_{z \in V_0 \cap R[x]} U_z \cup V_1 (w \Vdash C)$, we obtain $w \Vdash A \rhd C$. 
Therefore $\J{2}_+$ is valid in $F$. 
\end{proof}

\subsection{The axiom scheme $\J{5}$}

We investigate $\J{5}$. 
\begin{description}
	\item [$\J{5}$] $\Diamond A \rhd A$. 
\end{description}

The following proposition is stated in Visser \cite{Vis88}. 

\begin{prop}[Visser \cite{Vis88}]\label{5P1}
Let $F = \langle W, R, \{S_x\}_{x \in W} \rangle$ be any $\IL^-$-frame. 
The following are equivalent: 
\begin{enumerate}
	\item $\J{5}$ is valid in $F$. 
	\item $\forall x, y, z \in W (x R y\ \&\ y R z \Rightarrow y S_x z)$. 
\end{enumerate}
\end{prop}

\begin{proof}

$(1 \Rightarrow 2)$: 
Assume that $\mathbf{J5}$ is valid in $F$. Suppose $x R y$ and $y R z$. 
Let $\Vdash$ be a satisfaction relation on $F$ such that for any $u \in W$, $u \Vdash p$ if and only if $u=z$ for some fixed propositional variable $p$. 
Then $y R z$ and $z \Vdash p$, and hence $y \Vdash \Diamond p$. 
Since $x R y$ and $x \Vdash \Diamond p \rhd p$, there exists $u \in W$ such that $y S_x u$ and $u \Vdash p)$. 
By the definition of $\Vdash$, we have $u = z$. 
Therefore $y S_x z$. 

$(2 \Rightarrow 1)$: 
Assume that $\forall x, y, z \in W (x R y\ \&\ y R z \Rightarrow y S_x z)$. 
Let $\Vdash$ be any satisfaction relation on $F$. 
Let $y \in W$ be any element such that $x R y$ and $y \Vdash \Diamond A$. 
Then there exists $z \in W$ such that $y R z$ and $z \Vdash A$. 
By the assumption, $y S_x z$ and hence we obtain $x \Vdash \Diamond A \rhd A$. 
That is, $\J{5}$ is valid in $F$. 
\end{proof}

\begin{prop}\label{5P2}
Let $F = \langle W, R, \{S_x\}_{x \in W} \rangle$ be any $\ILS$-frame. 
Then the following are equivalent: 
\begin{enumerate}
	\item $\J{5}$ is valid in $F$.  
	\item $\forall x, y, z \in W(x R y\ \&\ y R z \Rightarrow y S_x \{z\})$. 
\end{enumerate}
\end{prop}
\begin{proof}
$(1 \Rightarrow 2)$: 
Assume that $\J{5}$ is valid in $F$. 
Suppose $x R y$ and $y R z$. 
Let $\Vdash$ be a satisfaction relation on $F$ such that for any $u \in W$, $u \Vdash p$ if and only if $u = z$ for some fixed propositional variable $p$.
Then $y R z$ and $z \Vdash p$, and hence $y \Vdash \Diamond p$. 
Since $x R y$ and $x \Vdash \Diamond p \rhd p$, there exists a $V \subseteq W$ such that $y S_x V$ and $\forall w \in V (w \Vdash p)$. 
By the definition of $\Vdash$, we have $V = \{z\}$. 
Therefore $y S_x \{z\}$. 

$(2 \Rightarrow 1)$: 
Assume $\forall x, y, z \in W(x R y\ \&\ y R z \Rightarrow y S_x \{z\})$. 
Let $y \in W$ be any element such that $x R y$ and $y \Vdash \Diamond A$. 
Then there exists $z \in W$ such that $y R z$ and $z \Vdash A$. 
By the assumption, $y S_x \{z\}$. 
Since $\forall w \in \{z\} (w \Vdash A)$, we obtain $x \Vdash \Diamond A \rhd A$. 
That is, $\J{5}$ is valid in $F$. 
\end{proof}

The condition stated in the second clause in Proposition \ref{5P2} is required in the original definition of $\IL_{\textrm{set}}$-frames.

\subsection{The logics $\CL$ and $\IL$}\label{SSec:CLIL}

In this subsection, we show that the logics $\CL$ and $\IL$ are exactly $\IL^-(\J{1}, \J{2})$ and $\IL^-(\J{1}, \J{2}, \J{5})$, respectively. 
Since $\IL^-(\J{1}, \J{2}, \J{5}) $ proves $\J{4}$, $\J{4}_+$ and $\J{2}_+$ by Propositions \ref{4P3}, \ref{2P2} and Corollary \ref{2C1}, our logics studied in this paper are actually sublogics of $\IL$. 
The logic $\CL$ is $\GL$ plus $\J{1}$, $\J{2}$, $\J{3}$ and $\J{4}$. 
Also the logic $\IL$ is $\CL$ plus $\J{5}$. 

\begin{prop}\label{CLP1}\leavevmode
\begin{enumerate}
	\item $\CL \vdash \Box A \leftrightarrow (\neg A) \rhd \bot$. 
	\item $\CL \vdash \Box(A \to B) \to (C \rhd A \to C \rhd B)$. 
	\item $\CL \vdash \Box(A \to B) \to (B \rhd C \to A \rhd C)$. 
\end{enumerate}
\end{prop}
\begin{proof}
1. $(\rightarrow)$: Since $\CL \vdash \Box A \to \Box (\neg A \to \bot)$, $\CL \vdash \Box A \to (\neg A) \rhd \bot$ by $\J{1}$. 

$(\leftarrow)$: By $\J{4}$, $\CL \vdash (\neg A) \rhd \bot \to (\Diamond \neg A \to \Diamond \bot)$. 
Since $\CL \vdash \neg \Diamond \bot$, $\CL \vdash (\neg A) \rhd \bot \to \neg \Diamond \neg A$. 
That is, $\CL \vdash (\neg A) \rhd \bot \to \Box A$. 

2. This is because $\CL \vdash \Box(A \to B) \to A \rhd B$ by $\J{1}$ and $\CL \vdash (C \rhd A) \land (A \rhd B) \to C \rhd B$ by $\J{2}$. 

3. This is because $\CL \vdash \Box(A \to B) \to A \rhd B$ by $\J{1}$ and $\CL \vdash (A \rhd B) \land (B \rhd C) \to A \rhd C$ by $\J{2}$. 
\end{proof}

\begin{prop}\label{CLP2}
The logics $\CL$ and $\IL^-(\J{1}, \J{2})$ are deductively equivalent. 
\end{prop}
\begin{proof}
$\CL \vdash \IL^-(\J{1}, \J{2})$: This follows from Proposition \ref{CLP1}. 

$\IL^-(\J{1}, \J{2}) \vdash \CL$: This is because $\IL^-(\J{2}) \vdash \J{4}$ by Corollary \ref{4C1}. 
\end{proof}

\begin{cor}\label{CLC1}
The logics $\IL$ and $\IL^-(\J{1}, \J{2}, \J{5})$ are deductively equivalent. 
\end{cor}

Then de Jongh and Veltman's and Ignatiev's theorems are restated as follows: 

\begin{thm}[de Jongh and Veltman \cite{deJVel90}]
For any formula $A$, the following are equivalent: 
\begin{enumerate}
	\item $\IL^-(\J{1}, \J{2}, \J{5}) \vdash A$. 
	\item $A$ is valid in all finite $\IL^-$-frames where all axioms of $\IL^-(\J{1}, \J{2}, \J{5})$ are valid. 
\end{enumerate}
\end{thm}

\begin{thm}[Ignatiev \cite{Ign91}]
For any formula $A$, the following are equivalent: 
\begin{enumerate}
	\item $\IL^-(\J{1}, \J{2}) \vdash A$. 
	\item $A$ is valid in all finite $\IL^-$-frames where all axioms of $\IL^-(\J{1}, \J{2})$ are valid. 
\end{enumerate}
\end{thm}

In the following, we identify $\CL$ with $\IL^-(\J{1}, \J{2})$, and $\IL$ with $\IL^-(\J{1}, \J{2}, \J{5})$.

\section{Lemmas for proofs of modal completeness theorems}\label{Sec:Lem}

In this section, we prepare some definitions and lemmas for our proofs of the modal completeness theorems of several logics. 
In this section, let $L$ be any consistent logic containing $\IL^-$. 
For a set $\Phi$ of formulas, define $\Phi_{\rhd} : = \{B :$ there exists a formula $C$ such that either $B \rhd C \in \Phi$ or $C \rhd B \in \Phi\}$. 
For each formula $A$, let ${\sim}A : \equiv \begin{cases} B & \text{if}\ A\ \text{is of the form}\ \neg B \\ \neg A & \text{otherwise} \end{cases}$. 
We say a finite set $\Gamma$ of formulas is \textit{$L$-consistent} if $L \nvdash \bigwedge \Gamma \to \bot$, where $\bigwedge \Gamma$ is a conjunction of all elements of $\Gamma$. 
Also we say $\Gamma \subseteq \Phi$ is \textit{$\Phi$-maximally $L$-consistent} if $\Gamma$ is $L$-consistent and for any $A \in \Phi$, either $A \in \Gamma$ or ${\sim}A \in \Gamma$. 
Notice that if $\Gamma$ is $\Phi$-maximally $L$-consistent and $L \vdash \bigwedge \Gamma \to A$ for $A \in \Phi$, then $A \in \Gamma$. 

\begin{defn}
A set $\Phi$ of formulas is said to be \textit{adequate} if it satisfies the following conditions: 
\begin{enumerate}
	\item $\Phi$ is closed under taking subformulas and applying $\sim$; 
	\item $\bot \in \Phi_\rhd$; 
	\item If $B, C \in \Phi_\rhd$, then $B \rhd C \in \Phi$; 
	\item If $B \in \Phi_\rhd$, then $\Box {\sim}B \in \Phi$; 
	\item If $B, C_1, \ldots, C_m, D_1, \ldots, D_n \in \Phi_\rhd$, then $\Box \left(B \to \bigvee_{i = 1}^m C_i \lor \bigvee_{j = 1}^n \Diamond D_j \right) \in \Phi$. 
\end{enumerate}
\end{defn}
Note that $\Box$ is in our language as a symbol, and $\Box A$ is not an abbreviation for $(\neg A) \rhd \bot$. Then the following proposition holds. 

\begin{prop}\label{PP1}
Every finite set of formulas is contained in some finite adequate set. 
\end{prop}

Until the end of this section, we fix some finite adequate set $\Phi$. 
Let $K_L : = \{\Gamma \subseteq \Phi : \Gamma$ is $\Phi$-maximally $L$-consistent$\}$. 
Then $K_L$ is also a finite set. 

\begin{defn}
Let $\Gamma, \Delta \in K_L$ and $C \in \Phi_\rhd$. 
\begin{enumerate}
	\item $\Gamma \prec \Delta : \iff$ 1. for any $\Box B \in \Phi$, if $\Box B \in \Gamma$, then $B, \Box B \in \Delta$ and 2. there exists $\Box B \in \Phi$ such that $\Box B \notin \Gamma$ and $\Box B \in \Delta$. 
	\item $\Gamma \prec_C \Delta : \iff \Gamma \prec \Delta$ and for any $B \in \Phi$, if $B \rhd C \in \Gamma$, then ${\sim}B \in \Delta$. 
	\item $\Gamma \prec_C^* \Delta : \iff \Gamma \prec \Delta$ and for any $B \in \Phi$, if $B \rhd C \in \Gamma$, then ${\sim}B, \Box {\sim}B \in \Delta$. 
\end{enumerate}
\end{defn}

The relation $\prec_C^*$ was introduced by de Jongh and Veltman \cite{deJVel90}, and $\Gamma \prec_C^* \Delta$ is read as ``$\Delta$ is a \textit{$C$-critical successor} of $\Gamma$''\footnote{For every set $S$ of formulas, more general notion of assuring successor $\prec_S$ was introduced and investigated in Goris et al.~\cite{GBJM20}. 
Then $\prec_C^*$ is exactly $\prec_{\{\neg C\}}$. 
However, in this paper, $\prec_C$ and $\prec_C^*$ are sufficient for our purpose. }. 
The relation $\prec_C$ was introduced by Ingatiev \cite{Ign91}. 
Obviously, $\Gamma \prec_C^* \Delta$ implies $\Gamma \prec_C \Delta$. 

\begin{lem}\label{PL1}
For $\Gamma, \Delta \in K_L$, if $\Gamma \prec \Delta$, then $\Gamma \prec_{\bot}^* \Delta$. 
\end{lem}
\begin{proof}
Suppose $\Gamma \prec \Delta$.  
If $B \rhd \bot \in \Gamma$, then $\Box {\sim}B \in \Gamma$ by $\J{6}$. 
Then ${\sim}B, \Box {\sim}B \in \Delta$. 
This means $\Gamma \prec_{\bot}^* \Delta$. 
\end{proof}

\begin{lem}\label{PL2}
Let $\Gamma, \Delta, \Theta \in K_L$ and $C \in \Phi_\rhd$. 
If $\Gamma \prec_C^* \Delta$ and $\Delta \prec \Theta$, then $\Gamma \prec_C^* \Theta$. 
\end{lem}
\begin{proof}
Suppose $\Gamma \prec_C^* \Delta$ and $\Delta \prec \Theta$. 
If $B \rhd C \in \Gamma$, then $\Box {\sim}B \in \Delta$. 
Then ${\sim}B, \Box {\sim}B \in \Theta$. 
Therefore $\Gamma \prec_C^* \Theta$. 
\end{proof}

\begin{lem}\label{PL3}
Let $\Gamma \in K_L$ and $D, E \in \Phi_\rhd$. 
If $D \rhd E \notin \Gamma$, then there exists $\Delta \in K_L$ such that $D \in \Delta$ and $\Gamma \prec_E \Delta$. 
Moreover: 
\renewcommand{\theenumi}{\alph{enumi}}
\begin{enumerate}
	\item If $L$ contains $\J{5}$, then we can find $\Delta$ such that in addition $\Box {\sim}E \in \Delta$ holds. 
	\item If $L$ contains $\J{2}$ and $\J{5}$, then we can find $\Delta$ such that in addition $\Gamma \prec_E^* \Delta$ and $\Box {\sim}E \in \Delta$ hold. 
\end{enumerate}
\renewcommand{\theenumi}{\arabic{enumi}}
\end{lem}
\begin{proof}
Suppose $D \rhd E \notin \Gamma$. 
Let $X : = \{G : G \rhd E \in \Gamma\}$. 
Then $\Box(D \to \bigvee X) \in \Phi$. 
By $\J{3}$, we have $\IL^- \vdash \bigwedge \Gamma \to \bigvee X \rhd E$. 

\begin{itemize}
\item Suppose, for the contradiction, that $\Box(D \to \bigvee X) \in \Gamma$. 
Then $\IL^- \vdash \bigwedge \Gamma \to (\bigvee X \rhd E \to D \rhd E)$ by Proposition \ref{0P3}.2. 
Hence $\IL^- \vdash \bigwedge \Gamma \to D \rhd E$, and thus $D \rhd E \in \Gamma$. 
This contradicts our supposition. 
Therefore $\Box(D \to \bigvee X) \notin \Gamma$. 

Let 
\[
	Y_0 :=\{B, \Box B : \Box B \in \Gamma\} \cup \{D, \Box(D \to \bigvee X)\} \cup \{{\sim}G : G \in X\}, 
\]
then $Y_0 \subseteq \Phi$. 
Suppose that the set $Y_0$ were $L$-inconsistent. 
Then for some $\Box B_1, \ldots, \Box B_k \in \Gamma$, 
\begin{align*}
	L & \vdash \bigwedge_{i = 1}^k (B_i \land \Box B_i) \to (\Box(D \to \bigvee X) \to (D \to \bigvee X)),\\
	L & \vdash \bigwedge_{i = 1}^k \Box B_i \to \Box (\Box(D \to \bigvee X) \to (D \to \bigvee X)),\\
	L & \vdash \bigwedge \Gamma \to \Box (D \to \bigvee X). 
\end{align*}
Thus $\Box (D \to \bigvee X) \in \Gamma$, and this is a contradiction. 
We have shown that $Y_0$ is $L$-consistent. 

Let $\Delta \in K_L$ be such that $Y_0 \subseteq \Delta$. 
Then $D \in \Delta$. 
Since $\Box(D \to \bigvee X) \in \Delta \setminus \Gamma$, $\Gamma \prec \Delta$. 
Moreover, if $G \rhd E \in \Gamma$, then $G \in X$, and hence ${\sim}G \in \Delta$. 
This means $\Gamma \prec_E \Delta$. 

\item a. Assume that $L$ contains $\J{5}$. 
Let $X_1 : = X \cup \{\Diamond E\}$. 
Then $\Box(D \to \bigvee X_1) \in \Phi$. 
If $\Box(D \to \bigvee X_1) \in \Gamma$, then $\IL^- \vdash \bigwedge \Gamma \to (\bigvee X_1 \rhd E \to D \rhd E)$. 
Since $\IL^- \vdash \bigwedge \Gamma \to \bigvee X \rhd E$ and $L \vdash \Diamond E \rhd E$ by $\J{5}$, we obtain $L \vdash \bigwedge \Gamma \to \bigvee X_1 \rhd E$. 
Thus $L \vdash \bigwedge \Gamma \to D \rhd E$ and $D \rhd E \in \Gamma$. 
This is a contradiction. 
Therefore $\Box(D \to \bigvee X_1) \notin \Gamma$. 

Let 
\[
	Y_1 :=\{B, \Box B : \Box B \in \Gamma\} \cup \{D, \Box(D \to \bigvee X_1)\} \cup \{{\sim}G : G \in X\} \cup \{\Box {\sim}E\}. 
\]
Then it can be proved that $Y_1$ is also an $L$-consistent subset of $\Phi$ as above. 
Let $\Delta \in K_L$ be such that $Y_1 \subseteq \Delta$. 
Then $\Delta$ satisfies the required conditions.  

\item b. Assume that $L$ contains $\J{2}$ and $\J{5}$. 
Let $X_2 : = X \cup \{\Diamond G : G \in X\} \cup \{\Diamond E\}$. 
Then $\Box(D \to \bigvee X_2) \in \Phi$. 
For each $G \in X$, we have $L \vdash \bigwedge \Gamma \to (\Diamond G \rhd G) \land (G \rhd E)$ by $\J{5}$. 
Then by $\J{2}$, $L \vdash \bigwedge \Gamma \to \Diamond G \rhd E$. 
Therefore we obtain $L \vdash \bigwedge \Gamma \to \bigwedge_{G \in X} (\Diamond G \rhd E)$. 
Since we also have $\IL^- \vdash \bigwedge \Gamma \to \bigvee X \rhd E$ and $L \vdash \Diamond E \rhd E$, we get $L \vdash \bigwedge \Gamma \to \bigvee X_2 \rhd E$. 
This implies $\Box(D \to \bigvee X_2) \notin \Gamma$. 

Let 
\[
	Y_2 :=\{B, \Box B : \Box B \in \Gamma\} \cup \{D, \Box(D \to \bigvee X_2)\} \cup \{{\sim}G, \Box {\sim}G : G \in X\} \cup \{\Box {\sim}E\}. 
\]
Then $Y_2$ is also an $L$-consistent subset of $\Phi$, and any $\Delta \in K_L$ with $Y_2 \subseteq \Delta$ is a desired set. 
\end{itemize}
\end{proof}

\begin{lem}\label{PL4}
Let $\Gamma, \Delta \in K_L$ and $D, E, F \in \Phi_\rhd$. 
If $D \rhd E \in \Gamma$, $\Gamma \prec_F \Delta$ and $D \in \Delta$, then there exists $\Theta \in K_L$ such that $E \in \Theta$ and ${\sim}F \in \Theta$. 
Moreover: 
\renewcommand{\theenumi}{\alph{enumi}}
\begin{enumerate}
	\item If $L$ contains $\J{4}_+$, then we can find $\Theta$ such that in addition $\Gamma \prec \Theta$ holds. 
	\item If $L$ contains $\J{2}_+$, then we can find $\Theta$ such that in addition $\Gamma \prec_F \Theta$ holds. 
	\item If $L$ contains $\J{2}_+$ and $\J{5}$, then we can find $\Theta$ such that in addition $\Gamma \prec_F^* \Theta$ and $\Box {\sim}F \in \Theta$ hold. 
\end{enumerate}
\renewcommand{\theenumi}{\arabic{enumi}}
\end{lem}
\begin{proof}
Suppose $D \rhd E \in \Gamma$, $\Gamma \prec_F \Delta$ and $D \in \Delta$. 
\begin{itemize}
	\item Suppose, towards a contradiction, that the set $\{E, {\sim}F\}$ is $L$-inconsistent. 
	Then $L \vdash E \to F$. 
	By the rule $\mathbf{R1}$, we have $L \vdash D \rhd E \to D \rhd F$, and hence $D \rhd F \in \Gamma$. 
	Since $\Gamma \prec_F \Delta$, we have ${\sim}D \in \Delta$. 
	This contradicts the $L$-consistency of $\Delta$. 
	Therefore $\{E, {\sim}F\}$ is $L$-consistent. 
	Let $\Theta \in K_L$ be such that $\{E, {\sim}F\} \subseteq \Theta$, and then $\Theta$ satisfies the required conditions. 
	
	\item a. Assume that $L$ contains $\J{4}_+$. 
If $\Box (E \to F) \in \Gamma$, then by $\J{4}_+$, $L \vdash \bigwedge \Gamma \to (D \rhd E \to D \rhd F)$. 
Since $D \rhd E \in \Gamma$, we obtain $D \rhd F \in \Gamma$. 
Then ${\sim}D \in \Delta$ because $\Gamma \prec_F \Delta$, and this is a contradiction. 
Therefore  $\Box (E \to F) \notin \Gamma$. 

Suppose $Y_1 : = \{B, \Box B : \Box B \in \Gamma\} \cup \{E, {\sim}F, \Box(E \to F)\}$ were $L$-inconsistent. 
Then there would be $\Box B_1, \ldots, \Box B_k \in \Gamma$ such that
\begin{align*}
	L & \vdash \bigwedge_{i=1}^k (B_i \land \Box B_i) \to (\Box (E \to F) \to (E \to F)),\\
	L & \vdash \bigwedge_{i=1}^k \Box B_i \to \Box (\Box (E \to F) \to (E \to F)),\\
	L & \vdash \bigwedge \Gamma \to \Box (E \to F). 
\end{align*}
Then $\Box (E \to F) \in \Gamma$, and this is a contradiction. 
Thus $Y_1$ is $L$-consistent. 
Let $\Theta \in K_L$ be such that $Y_1 \subseteq \Theta$. 
Then $\Box(E \to F) \in \Theta \setminus \Gamma$, and hence we conclude $\Gamma \prec \Theta$. 

	\item b. Assume that $L$ contains $\J{2}_+$. 
Let $X : = \{G : G \rhd F \in \Gamma\}$. 
If $\Box (E \to \bigvee X \lor F) \in \Gamma$, then $L \vdash \bigwedge \Gamma \to \Box (E \land \neg F \to \bigvee X)$, and hence $L \vdash \bigwedge \Gamma \to (\bigvee X \rhd F \to (E \land \neg F) \rhd F)$ by Proposition \ref{0P3}.2. 
Since $L \vdash \bigwedge \Gamma \to \bigvee X \rhd F$, we have $L \vdash \bigwedge \Gamma \to (E \land \neg F) \rhd F$. 
Since $D \rhd E \in \Gamma$, $L \vdash \bigwedge \Gamma \to D \rhd F$ by $\J{2}_+'$. 
Thus $D \rhd F \in \Gamma$. 
Then ${\sim}F \in \Delta$ because $\Gamma \prec_F \Delta$, and this is a contradiction. Hence $\Box (E \to \bigvee X \lor F) \notin \Gamma$. 

Let 
\[
	Y_2 : = \{B, \Box B : \Box B \in \Gamma\} \cup \{E, {\sim}F, \Box(E \to \bigvee X \lor F)\} \cup \{{\sim}G : G \in X\}, 
\]
then $Y_2$ is $L$-consistent. 
Let $\Theta \in K_L$ be such that $Y_2 \subseteq \Theta$. 
Then $\Box(E \to \bigvee X \lor F) \in \Theta \setminus \Gamma$, and hence $\Gamma \prec \Theta$. 
Moreover, $\Gamma \prec_F \Theta$. 

	\item c. Assume that $L$ contains $\J{2}_+$ and $\J{5}$. 
Let $X : = \{G : G \rhd F \in \Gamma\}$ and $X_1 : = X \cup \{\Diamond G : G \in X\} \cup \{\Diamond F\}$. 
Then $\Box (E \to \bigvee X_1 \lor F) \in \Phi$. 
For each $G \in X$, $L \vdash \bigwedge \Gamma \to (\Diamond G \rhd G) \land (G \rhd F)$ by $\J{5}$. 
Then $L \vdash \bigwedge \Gamma \to \Diamond G \rhd F$ by $\J{2}$. 
Since $L \vdash \bigwedge \Gamma \to \bigvee X \rhd F$ and $L \vdash \Diamond F \rhd F$, we obtain $L \vdash \bigwedge \Gamma \to \bigvee X_1 \rhd F$. 

Suppose, towards a contradiction, $\Box (E \to \bigvee X_1 \lor F) \in \Gamma$. 
Then $L \vdash \bigwedge \Gamma \to \Box (E \land \neg F \to \bigvee X_1)$, and thus $L \vdash \bigwedge \Gamma \to (\bigvee X_1 \rhd F \to (E \land \neg F) \rhd F)$ by Proposition \ref{0P3}.2. 
Hence $L \vdash \bigwedge \Gamma \to (E \land \neg F) \rhd F$. 
Since $D \rhd E \in \Gamma$, by $\J{2}_+'$, we have $L \vdash \bigwedge \Gamma \to D \rhd F$. 
Thus $D \rhd F \in \Gamma$. 
Then ${\sim}F \in \Delta$ because $\Gamma \prec_F \Delta$. 
This is a contradiction. 
Hence we obtain $\Box (E \to \bigvee X_1 \lor F) \notin \Gamma$. 

Let 
\[
	Y_3 : = \{B, \Box B : \Box B \in \Gamma\} \cup \{E, {\sim}F, \Box(E \to \bigvee X_1 \lor F)\} \cup \{{\sim}G, \Box {\sim}G : G \in X\} \cup \{\Box {\sim}F\}, 
\]
then we can prove that $Y_3$ is $L$-consistent. 
Let $\Theta \in K_L$ be such that $Y_3 \subseteq \Theta$. 
Then $\Box(E \to \bigvee X_1 \lor F) \in \Theta \setminus \Gamma$, and hence $\Gamma \prec \Theta$. 
Moreover, $\Gamma \prec_F^* \Theta$. 
\end{itemize}
\end{proof}

\begin{lem}\label{PL5}
Assume that $L$ contains $\J{4}$. 
Let $\Gamma, \Delta \in K_L$ and $D, E \in \Phi_\rhd$. 
If $D \rhd E \in \Gamma$, $\Gamma \prec \Delta$ and $D \in \Delta$, then there exists $\Theta \in K_L$ such that $\Gamma \prec \Theta$ and $E \in \Theta$. 
\end{lem}
\begin{proof}
Since $D \rhd E \in \Gamma$, $L \vdash \bigwedge \Gamma \to (\Diamond D \to \Diamond E)$ by $\J{4}$. 
If $\Box {\sim}E \in \Gamma$, then $\Box {\sim}D \in \Gamma$. 
Since $\Gamma \prec \Delta$, we have ${\sim}D \in \Delta$, a contradiction. 
Thus $\Box {\sim}E \notin \Gamma$. 

Let $Y : = \{B, \Box B : \Box B \in \Gamma\} \cup \{E, \Box {\sim}E\}$, then it is proved that $Y$ is $L$-consistent. 
Thus for some $\Theta \in K_L$, $Y \subseteq \Theta$. 
Since $\Box {\sim}E \in \Theta \setminus \Gamma$, we obtain $\Gamma \prec \Theta$. 
\end{proof}

\begin{lem}\label{PL6}
Assume that $L$ contains $\J{2}$. 
Let $\Gamma, \Delta \in K_L$ and $D, E, F \in \Phi_\rhd$. 
If $D \rhd E \in \Gamma$, $\Gamma \prec_F \Delta$ and $D \in \Delta$, then there exists $\Theta \in K_L$ such that $\Gamma \prec_F \Theta$ and $E \in \Theta$. 
Moreover:
\renewcommand{\theenumi}{\alph{enumi}}
\begin{enumerate}
	\item If $L$ contains $\J{5}$, then we can find $\Theta$ such that in addition  $\Gamma \prec_F^* \Theta$ and $\Box {\sim}F \in \Theta$ hold. 
\end{enumerate}
\renewcommand{\theenumi}{\arabic{enumi}}
\end{lem}
\begin{proof}
Let $X : = \{G : G \rhd F \in \Gamma\}$. 
Suppose, towards a contradiction, that $\Box (E \to \bigvee X) \in \Gamma$. 
Then by Proposition \ref{0P3}.2, $\IL^- \vdash \bigwedge \Gamma \to (\bigvee X \rhd F \to E \rhd F)$. 
Since $\IL^- \vdash \bigwedge \Gamma \to \bigvee X \rhd F$, $\IL^- \vdash \bigwedge \Gamma \to E \rhd F$. 
Also since $D \rhd E \in \Gamma$, we obtain $L \vdash \bigwedge \Gamma \to D \rhd F$ by $\J{2}$. 
Thus $D \rhd F \in \Gamma$. 
Since $\Gamma \prec_F \Delta$, ${\sim}D \in \Delta$ and hence this contradicts the $L$-consistency of $\Delta$. 
Therefore $\Box (E \to \bigvee X) \notin \Gamma$. 

Let $Y_0 : = \{B, \Box B : \Box B \in \Gamma\} \cup \{E, \Box(E \to \bigvee X)\} \cup \{{\sim}G : G \in X\}$, then $Y_0$ is $L$-consistent. 
Let $\Theta \in K_L$ be such that $Y_0 \subseteq \Theta$. 
Then $\Theta$ is a desired set. 

a. Assume that $L$ contains $\J{5}$. 
Let $X_1 : = X \cup \{\Diamond G : G \in X\} \cup \{\Diamond F\}$. 
If $\Box (E \to \bigvee X_1) \in \Gamma$, then $E \rhd F \in \Gamma $ as in the proof of Lemma \ref{PL3}.b. 
Since $D \rhd E \in \Gamma$, we have $D \rhd F \in \Gamma$ by applying $\J{2}$.  
This contradicts $\Gamma \prec_F \Theta$ and $E \in \Theta$. 
Therefore $\Box (E \to \bigvee X_1)$ is not in $\Gamma$. 
Let $Y_1 : = \{B, \Box B : \Box B \in \Gamma\} \cup \{E, \Box(E \to \bigvee X_1)\} \cup \{{\sim}G, \Box {\sim}G : G \in X\} \cup \{\Box {\sim}F\}$. 
Then $Y_1$ is $L$-consistent. 
Let $\Theta \in K_L$ be such that $Y_1 \subseteq \Theta$. 
\end{proof}

\begin{table}[ht]
\centering
\begin{tabular}{|l||c|c|c|c|c||l|}
\hline
 & $E \in \Theta$ & ${\sim}F \in \Theta$ & $\Gamma \prec \Theta$ & $\Gamma \prec_F \Theta$ & \begin{tabular}{c} $\Gamma \prec_F^* \Theta$ \\ \& $\Box {\sim}F \in \Theta$ \end{tabular} & \\
\hline
 & $\checkmark$ & $\checkmark$ & & & & Lemma \ref{PL4} \\
\hline
$\J{4}$ & $\checkmark$ & & $\checkmark$ & & & Lemma \ref{PL5} \\
\hline
$\J{4}_+$ & $\checkmark$ & $\checkmark$ & $\checkmark$ & & & Lemma \ref{PL4}.a\\
\hline
$\J{2}$ & $\checkmark$ & & & $\checkmark$ & & Lemma \ref{PL6}\\
\hline
$\J{2}$, $\J{5}$ & $\checkmark$ & & & & $\checkmark$ & Lemma \ref{PL6}.a\\
\hline
$\J{2}_+$ & $\checkmark$ & $\checkmark$ & & $\checkmark$ & & Lemma \ref{PL4}.b\\
\hline
$\J{2}_+$, $\J{5}$ & $\checkmark$ & $\checkmark$ & & & $\checkmark$ & Lemma \ref{PL4}.c \\
\hline
\end{tabular}
\caption{Conclusions of Lemmas \ref{PL4}, \ref{PL5} and \ref{PL6}}\label{Tab1}
\end{table}

Lemmas \ref{PL4} and \ref{PL6} state that if $D \rhd E$, $\Gamma \prec_F \Delta$ and $D \in \Delta$, then there exists $\Theta \in K_L$ having several properties depending on each logic $L$. 
The statement of Lemma \ref{PL5} is similar except that $\Gamma \prec \Delta$ is assumed instead of $\Gamma \prec_F \Delta$. 
To compare the properties that $\Theta$ is assured to have, we summarize conclusions of these lemmas in Table \ref{Tab1}. 
For example, the fifth line of the table shows that if $L$ contains $\J{2}$, then such a set $\Theta$ satisfying $E \in  \Theta$ and $\Gamma \prec_F \Theta$ is obtained by Lemma \ref{PL6}.
Note that the assumptions of each lemma are omitted in the table for the sake of simplicity.

\section{Modal completeness with respect to $\IL^-$-frames}\label{Sec:Compl}

In the previous sections, we dealt with the additional axioms $\J{1}$, $\J{2}$, $\J{2}_{+}$, $\J{4}$, $\J{4}_{+}$ and $\J{5}$. 
From Corollary \ref{4C1} and Propositions \ref{4P3}, \ref{2P2} and \ref{2P3}, we know that there are twenty different logics obtained by adding some of these axioms to $\IL^-$. 
In this section, we prove modal completeness theorems with respect to $\IL^-$-frames for twelve of them. 
Figure \ref{Fig1} represents the interrelations between these twelve logics. 
In the figure, each line segment shows that the logic on the right side is a proper extension of the logic on the left side, where properness comes from our investigations in Section  \ref{Sec:Ext}. 
It follows that no more line segments can be drawn in the figure. 
The remaining eight logics are investigated in Sections \ref{Sec:Incompl} and \ref{Sec:GCompl}. 

\begin{figure}[ht]
\centering
\begin{tikzpicture}
\node (IL-) at (0,1.5) {$\IL^-$};
\node (IL5) at (3,0) {$\IL^-(\J{5})$};
\node (IL1) at (3,1.5) {$\IL^-(\J{1})$};
\node (IL4) at (3,3){$\IL^-(\J{4}_+)$};
\node (IL15) at (6,0){$\IL^-(\J{1}, \J{5})$};
\node (IL45) at (6,1.5){$\IL^-(\J{4}_+, \J{5})$};
\node (IL14) at (6,3){$\IL^-(\J{1}, \J{4}_+)$};
\node (IL2) at (6,4.5){$\IL^-(\J{2}_+)$};
\node (IL145) at (9,1.5){$\IL^-(\J{1}, \J{4}_+, \J{5})$};
\node (IL25) at (9,3){$\IL^-(\J{2}_+, \J{5})$};
\node (CL) at (9,4.5){$\CL$};
\node (IL) at (12,3){$\IL$};
\draw [-] (IL5)--(IL-);
\draw [-] (IL1)--(IL-);
\draw [-] (IL4)--(IL-);
\draw [-] (IL15)--(IL5);
\draw [-] (IL45)--(IL5);
\draw [-] (IL15)--(IL1);
\draw [-] (IL14)--(IL1);
\draw [-] (IL45)--(IL4);
\draw [-] (IL14)--(IL4);
\draw [-] (IL2)--(IL4);
\draw [-] (IL145)--(IL15);
\draw [-] (IL145)--(IL45);
\draw [-] (IL25)--(IL45);
\draw [-] (IL145)--(IL14);
\draw [-] (CL)--(IL14);
\draw [-] (IL25)--(IL2);
\draw [-] (CL)--(IL2);
\draw [-] (IL)--(IL145);
\draw [-] (IL)--(IL25);
\draw [-] (IL)--(CL);
\end{tikzpicture}
\caption{Sublogics of $\IL$ complete with respect to $\IL^-$-frames}
\label{Fig1}
\end{figure}
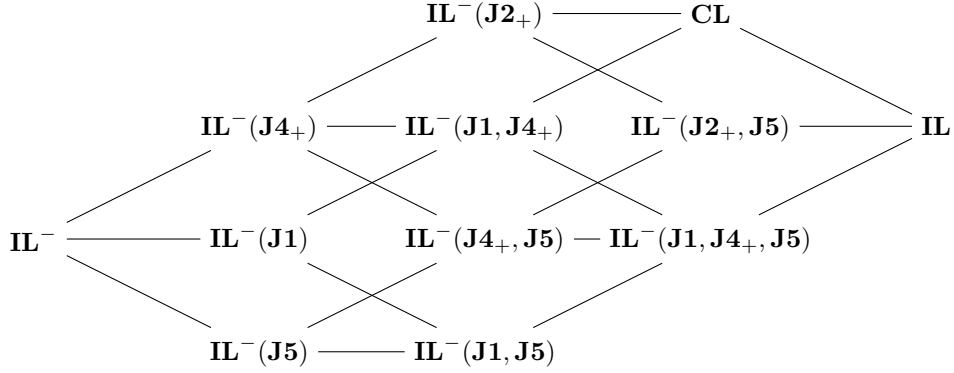

First, we prove the completeness theorem for logics in Figure \ref{Fig1} other than $\IL^-(\J{2}_+, \J{5})$ and $\IL$. 
Secondly, we prove the completeness theorem for logics $\IL^-(\J{2}_+, \J{5})$ and $\IL$. 
Our proof technique of the second completeness theorem is essentially same as in the proof of de Jongh and Veltman \cite{deJVel90}. 
However, the detail of our proof is different from that of proofs presented in \cite{deJVel90} and \cite{JapdeJ98}.
Furthermore, our proof of the first completeness theorem admits a simpler technique than that of the second theorem. 
More precisely, in the proof of the second theorem, the universe of a countermodel is defined as a set of tuples $\langle \Gamma, \tau \rangle$ where $\Gamma$ is a $\Phi$-maximal $L$-consistent subset of a finite adequate set $\Phi$ and $\tau$ is a finite sequence of formulas in $\Phi$. 
On the other hand, in our proof of the first theorem, we simply consider tuples $\langle \Gamma, B \rangle$ where $B$ is a formula in $\Phi$ to define a countermodel. 
As a consequence, our proof of the completeness theorem of the logic $\CL$ is simpler than Ignatiev's proof in \cite{Ign91}. 

First, we prove the completeness theorem for logics other than $\IL^-(\J{2}_+, \J{5})$ and $\IL$. 

\begin{thm}\label{CT1}
Let $L$ be one of the logics $\IL^-$, $\IL^-(\J{4}_+)$, $\IL^-(\J{1})$, $\IL^-(\J{5})$, $\IL^-(\J{2}_+)$, $\IL^-(\J{1}, \J{4}_+)$, $\IL^-(\J{4}_+, \J{5})$, $\IL^-(\J{1}, \J{5})$, $\CL$ and $\IL^-(\J{1}, \J{4}_+, \J{5})$. 
Then for any formula $A$, the following are equivalent: 
\begin{enumerate}
	\item $L \vdash A$. 
	\item $A$ is valid in all (finite) $\IL^-$-frames where all axioms of $L$ are valid. 
\end{enumerate}
\end{thm}
\begin{proof}
$(1 \Rightarrow 2)$: Obvious. 

$(2 \Rightarrow 1)$: 
Suppose $L \nvdash A$. 
Let $\Phi$ be a finite adequate set of formulas with ${\sim}A \in \Phi$. 
The existence of such a set $\Phi$ is guaranteed by Proposition \ref{PP1}. 
By the supposition, there exists $\Gamma_0 \in K_L$ such that ${\sim}A \in \Gamma_0$. 

Let $M = \langle W, R, \{S_x\}_{x \in W}, \Vdash \rangle$ be a model satisfying the following clauses: 
\begin{enumerate}
	\item $W = \{\langle \Gamma, B \rangle : \Gamma \in K_L$ and $B \in \Phi_\rhd\}$; 
	\item $\langle \Gamma, B \rangle R \langle \Delta, C \rangle \iff \Gamma \prec \Delta$; 
	\item $\langle \Delta, C \rangle S_{\langle \Gamma, B \rangle} \langle \Theta, D \rangle \iff \langle \Gamma, B \rangle R \langle \Delta, C \rangle$ and the condition $\mathcal{C}_L$ which is defined below holds; 
	\item $\langle \Gamma, B \rangle \Vdash p \iff p \in \Gamma$. 
\end{enumerate}

The condition $\mathcal{C}_L$ depends on $L$ as follows:

\begin{itemize}
	\item $L \in \{\IL^-, \IL^-(\J{1})\}$: If $\Gamma \prec_C \Delta$, then ${\sim}C \in \Theta$. 
	\item $L \in \{\IL^-(\J{4}_+), \IL^-(\J{1}, \J{4}_+)\}$: $\langle \Gamma, B \rangle R \langle \Theta, D \rangle$ and if $\Gamma \prec_C \Delta$, then ${\sim}C \in \Theta$. 
	\item $L \in \{\IL^-(\J{2}_+), \CL\}$: $\langle \Gamma, B \rangle R \langle \Theta, D \rangle$ and if $\Gamma \prec_C \Delta$, then $D \equiv C$, $\Gamma \prec_C \Theta$ and ${\sim}C \in \Theta$. 
	\item $L \in \{\IL^-(\J{5}), \IL^-(\J{1}, \J{5})\}$: If $\Gamma \prec_C \Delta$ and $\Box {\sim}C \in \Delta$, then ${\sim}C \in \Theta$. 
	\item $L \in \{\IL^-(\J{4}_+, \J{5}), \IL^-(\J{1}, \J{4}_+, \J{5})\}$: $\langle \Gamma, B \rangle R \langle \Theta, D \rangle$ and if $\Gamma \prec_C \Delta$ and $\Box {\sim}C \in \Delta$, then ${\sim}C \in \Theta$. 
\end{itemize}

Here $D \equiv C$ means that formulas $D$ and $C$ are identical. 
Since $\bot \in \Phi_\rhd$, we have $\langle \Gamma_0, \bot \rangle \in W$ and therefore $W$ is non-empty. 
Also $W$ is finite and $R$ is a transitive and conversely well-founded binary relation on $W$. 
Thus $\langle W, R, \{S_x\}_{x \in W} \rangle$ is an $\IL^-$-frame.

\begin{lem}
Every axiom of $L$ is valid in the frame $F = \langle W, R, \{S_x\}_{x \in W} \rangle$ of $M$. 
\end{lem}
\begin{proof}
We distinguish the following several cases: 
\begin{itemize}
	\item $L = \IL^-(\J{1})$: Suppose $\langle \Gamma, B \rangle R \langle \Delta, C \rangle$. 
	If $\Gamma \prec_C \Delta$, then ${\sim}C \in \Delta$ because $C \rhd C \in \Gamma$. 
Thus $\langle \Delta, C \rangle S_{\langle \Gamma, B \rangle} \langle \Delta, C \rangle$ by the definition of $\mathcal{C}_L$. 
Therefore $\J{1}$ is valid in $F$ by Proposition \ref{1P2}. 

	\item $L = \IL^-(\J{4}_+)$: If $\langle \Delta, C \rangle S_{\langle \Gamma, B \rangle} \langle \Theta, D \rangle$, then $\langle \Gamma, B \rangle R \langle \Theta, D \rangle$. 
By Proposition \ref{4P4}, $\J{4}_+$ is valid in $F$. 

	\item $L = \IL^-(\J{2}_+)$: As in the case of $\IL^-(\J{4}_+)$, $\J{4}_+$ is valid in $F$. 
	Suppose $\langle \Delta_0, C_0 \rangle S_{\langle \Gamma, B \rangle} \langle \Delta_1, C_1 \rangle$ and $\langle \Delta_1, C_1 \rangle S_{\langle \Gamma, B \rangle} \langle \Delta_2, C_2 \rangle$. 
Then $\langle \Gamma, B \rangle R \langle \Delta_2, C_2 \rangle$. 
If $\Gamma \prec_{C_0} \Delta_0$, then $C_1 \equiv C_0$ and $\Gamma \prec_{C_0} \Delta_1$. 
Since $\langle \Delta_1, C_0 \rangle S_{\langle \Gamma, B \rangle} \langle \Delta_2, C_2 \rangle$ and $\Gamma \prec_{C_0} \Delta_1$, we have $C_2 \equiv C_0$, $\Gamma \prec_{C_0} \Delta_2$ and ${\sim}C_0 \in \Delta_2$. 
Thus we obtain $\langle \Delta_0, C_0 \rangle S_{\langle \Gamma, B \rangle} \langle \Delta_2, C_2 \rangle$. 
Therefore $\J{2}_+$ is valid in $F$ by Proposition \ref{2P4}. 

	\item $L = \IL^-(\J{5})$: Suppose $\langle \Gamma, B \rangle R \langle \Delta, C \rangle$ and $\langle \Delta, C \rangle R \langle \Theta, D \rangle$. 
If $\Gamma \prec_C \Delta$ and $\Box {\sim}C \in \Delta$, then ${\sim}C \in \Theta$ because $\Delta \prec \Theta$. 
Thus $\langle \Delta, C \rangle S_{\langle \Gamma, B \rangle} \langle \Theta, D \rangle$ holds. 
Then by Proposition \ref{5P1}, $\J{5}$ is valid in $F$. 

\item For other cases, the lemma is proved in a similar way as above. 
\end{itemize}
\end{proof}

\begin{lem}[Truth Lemma]
For any formula $C \in \Phi$ and any $\langle \Gamma, B \rangle \in W$, $C \in \Gamma$ if and only if $\langle \Gamma, B \rangle \Vdash C$. 
\end{lem}
\begin{proof}
We prove by induction on the construction of $C$. 
We only give a proof of the case $C \equiv (D \rhd E)$. 

$(\Rightarrow)$: 
Assume $D \rhd E \in \Gamma$. 
Let $\langle \Delta, F \rangle$ be any element of $W$ such that $\langle \Gamma, B \rangle R \langle \Delta, F \rangle$ and $\langle \Delta, F \rangle \Vdash D$. 
Then by induction hypothesis, $D \in \Delta$. 
We distinguish the following two cases. 

\begin{itemize}
	\item If $\Gamma \prec_F \Delta$, then by Lemma \ref{PL4}, there exists $\Theta \in K_L$ such that $E \in \Theta$ and ${\sim}F \in \Theta$. 
	Moreover, if $L \vdash \J{4}_+$, $\Gamma \prec \Theta$ holds. 
	Also if $L \vdash \J{2}_+$, $\Gamma \prec_F \Theta$ holds. 

	\item If $\Gamma \nprec_F \Delta$, by Lemma \ref{PL4}, there exists $\Theta \in K_L$ such that $E \in \Theta$ because $\Gamma \prec_\bot \Delta$. 
	Moreover, if $L \vdash \J{4}_+$, $\Gamma \prec \Theta$ holds. 
\end{itemize}

	In either case, we have $\langle \Theta, F \rangle \in W$. 
	Also $E \in \Theta$ and $\langle \Delta, F \rangle S_{\langle \Gamma, B \rangle} \langle \Theta, F \rangle$. 
	Then by induction hypothesis, $\langle \Theta, F \rangle \Vdash E$. 
	Therefore we conclude $\langle \Gamma, B \rangle \Vdash D \rhd E$.

$(\Leftarrow)$: 
Assume $D \rhd E \notin \Gamma$. 
	By Lemma \ref{PL3}, there exists $\Delta \in K_L$ such that $D \in \Delta$ and $\Gamma \prec_E \Delta$. 
	Moreover, if $L$ contains $\J{5}$, then $\Box {\sim}E \in \Delta$ also holds. 
	Since $\langle \Delta, E \rangle \in W$, $\langle \Delta, E \rangle \Vdash D$ by induction hypothesis. 
	Let $\langle \Theta, F \rangle$ be any element of $W$ with $\langle \Delta, E \rangle S_{\langle \Gamma, B \rangle} \langle \Theta, F \rangle$. 
	By the definitions of the relations $S_{\langle \Gamma, B \rangle}$ and the condition $\mathcal{C}_L$, we have ${\sim}E \in \Theta$ in all cases of $L$. 
	By induction hypothesis, $\langle \Theta, F \rangle \nVdash E$. 
	Therefore we obtain $\langle \Gamma, B \rangle \nVdash D \rhd E$. 
\end{proof}

Since $\langle \Gamma_0, \bot \rangle \in W$ and $A \notin \Gamma_0$, $\langle \Gamma_0, \bot \rangle \nVdash A$ by Truth Lemma. 
Therefore $A$ is not valid in the frame of $M$. 
\end{proof}

Our proof of Theorem \ref{CT1} cannot be applied to logics containing both $\J{2}$ and $\J{5}$. 
For example, for $L = \IL^-(\J{2}_+, \J{5})$, the condition $\mathcal{C}_L$ which is used to define the relations $S_{\langle \Gamma, B \rangle}$ might be as follows: 
$\langle \Gamma, B \rangle R \langle \Theta, D \rangle$ and if $\Gamma \prec_C \Delta$ and $\Box {\sim}C \in \Delta$, then $D \equiv C$, $\Gamma \prec_C \Theta$ and ${\sim}C \in \Theta$. 
Then $\J{5}$ is no longer valid in the resulting frame $\langle W, R, \{S_x\}_{x \in W} \rangle$. 
To avoid this obstacle, as mentioned above, for the modal completeness of such logics, we consider tuples $\langle \Gamma, \tau \rangle$ as members of the universe of our countermodel, where $\tau$ is a finite sequence of formulas. 

For finite sequences $\tau$ and $\sigma$ of formulas, $\tau \subseteq \sigma$ denotes that $\tau$ is an initial segment of $\sigma$. 
Also $\tau \subsetneq \sigma$ denotes that $\tau$ is a proper initial segment of $\sigma$, that is, $\tau \subseteq \sigma$ and $|\tau| < |\sigma|$, where $|\tau|$ is the length of $\tau$. 
Let $\tau \ast \langle B \rangle$ be the sequence obtained from $\tau$ by concatenating $B$ as the last element. 

\begin{thm}\label{CT2}
Let $L$ be one of the logics $\IL^-(\J{2}_+, \J{5})$ and $\IL$. 
Then for any formula $A$, the following are equivalent: 
\begin{enumerate}
	\item $L \vdash A$. 
	\item $A$ is valid in all (finite) $\IL^-$-frames where all axioms of $L$ are valid. 
\end{enumerate}
\end{thm}
\begin{proof}
$(1 \Rightarrow 2)$: Obvious. 

$(2 \Rightarrow 1)$: 
Suppose $L \nvdash A$. 
Let $\Phi$ be any finite adequate set with ${\sim}A \in \Phi$. 
Let $\Gamma_0 \in K_L$ be such that ${\sim}A \in \Gamma_0$. 

For each $\Gamma \in K_L$, we define the rank of $\Gamma$ (write $\rank(\Gamma)$) as follows: $\rank(\Gamma) : = \sup\{\rank(\Delta) + 1 : \Gamma \prec \Delta\}$, where $\sup \emptyset = 0$. 
This is well-defined because $\prec$ is conversely well-founded. 

Let $M = \langle W, R, \{S_x\}_{x \in W}, \Vdash \rangle$ be a model satisfying the following clauses: 
\begin{enumerate}
	\item $W = \{\langle \Gamma, \tau \rangle : \Gamma \in K_L$ and $\tau$ is a finite sequence of elements of $\Phi_\rhd$ with $\rank(\Gamma) + |\tau| \leq \rank(\Gamma_0)\}$; 
	\item $\langle \Gamma, \tau \rangle R \langle \Delta, \sigma \rangle \iff \Gamma \prec \Delta$ and $\tau \subsetneq \sigma$; 
	\item $\langle \Delta, \sigma \rangle S_{\langle \Gamma, \tau \rangle} \langle \Theta, \rho \rangle \iff \langle \Gamma, \tau \rangle R \langle \Delta, \sigma \rangle$, $\langle \Gamma, \tau \rangle R \langle \Theta, \rho \rangle$ and if $\tau \ast \langle C \rangle \subseteq \sigma$, $\Gamma \prec_C^* \Delta$ and $\Box {\sim}C \in \Delta$, then $\tau \ast \langle C \rangle \subseteq \rho$, $\Gamma \prec_C^* \Theta$ and ${\sim}C, \Box {\sim}C \in \Theta$; 
	\item $\langle \Gamma, \tau \rangle \Vdash p \iff p \in \Gamma$. 
\end{enumerate}

Let $\epsilon$ be the empty sequence. 
Then $\rank(\Gamma_0) + |\epsilon| = \rank(\Gamma_0)$, and hence $\langle \Gamma_0, \epsilon \rangle \in W$. 
Therefore $W$ is a non-empty set. 
Also $W$ is finite because of the condition $\rank(\Gamma) + |\tau| \leq \rank(\Gamma_0)$. 
Then $\langle W, R, \{S_x\}_{x \in W} \rangle$ is an $\IL^-$-frame.

\begin{lem}
Every axiom of $L$ is valid in the frame $F = \langle W, R, \{S_x\}_{x \in W} \rangle$ of $M$. 
\end{lem}
\begin{proof}
$\J{2}_+$: 
By the definition of $S$, $\J{4}$ is obviously valid in $F$. 
Assume $\langle \Delta_0, \sigma_0 \rangle S_{\langle \Gamma, \tau \rangle} \langle \Delta_1, \sigma_1 \rangle$ and $\langle \Delta_1, \sigma_1 \rangle S_{\langle \Gamma, \tau \rangle} \langle \Delta_2, \sigma_2 \rangle$. 
Suppose $\tau \ast \langle C \rangle \subseteq \sigma_0$, $\Gamma \prec_C^* \Delta_0$ and $\Box {\sim}C \in \Delta_0$. 
Then $\tau \ast \langle C \rangle \subseteq \sigma_1$, $\Gamma \prec_C^* \Delta_1$ and $\Box {\sim}C \in \Delta_1$ because $\langle \Delta_0, \sigma_0 \rangle S_{\langle \Gamma, \tau \rangle} \langle \Delta_1, \sigma_1 \rangle$. 
Then also $\tau \ast \langle C \rangle \subseteq \sigma_2$, $\Gamma \prec_C^* \Delta_2$ and ${\sim}C, \Box {\sim}C \in \Delta_2$ because $\langle \Delta_1, \sigma_1 \rangle S_{\langle \Gamma, \tau \rangle} \langle \Delta_2, \sigma_2 \rangle$.  
Thus we obtain $\langle \Delta_0, \sigma_0 \rangle S_{\langle \Gamma, \tau \rangle} \langle \Delta_2, \sigma_2 \rangle$. 
Therefore $\J{2}_+$ is valid in $F$ by Proposition \ref{2P6}. 

$\J{5}$: 
Assume that $\langle \Gamma, \tau \rangle R \langle \Delta, \sigma \rangle$ and $\langle \Delta, \sigma \rangle R \langle \Theta, \rho \rangle$. 
Suppose $\tau \ast \langle C \rangle \subseteq \sigma$, $\Gamma \prec_C^* \Delta$ and $\Box {\sim}C \in \Delta$. 
	Since $\sigma \subsetneq \rho$, we have $\tau \ast \langle C \rangle \subseteq \rho$. 
	Since $\Gamma \prec_C^* \Delta$ and $\Delta \prec \Theta$, $\Gamma \prec_C^* \Theta$ by Lemma \ref{PL2}. 
	Also we have ${\sim}C, \Box {\sim}C \in \Theta$ because $\Delta \prec \Theta$. 
	Therefore we obtain $\langle \Delta, \sigma \rangle S_{\langle \Gamma, \tau \rangle} \langle \Theta, \rho \rangle$. 
	By Proposition \ref{5P1}, $\J{5}$ is valid in $F$. 

At last, we assume $L = \IL$ and show that $\J{1}$ is valid in $F$. 
Suppose $\langle \Gamma, \tau \rangle R \langle \Delta, \sigma \rangle$, $\Gamma \prec_C^* \Delta$ and $\Box {\sim}C \in \Delta$. 
Since $C \rhd C \in \Gamma$, ${\sim}C \in \Delta$. 
Thus we have $\langle \Delta, \sigma \rangle S_{\langle \Gamma, \tau \rangle} \langle \Delta, \sigma \rangle$. 
By Proposition \ref{1P2}, $\J{1}$ is valid in $F$. 
\end{proof}

\begin{lem}[Truth Lemma]
For any formula $C \in \Phi$ and any $\langle \Gamma, \tau \rangle \in W$, $C \in \Gamma$ if and only if $\langle \Gamma, \tau \rangle \Vdash C$. 
\end{lem}
\begin{proof}
This is proved by induction on the construction of $C$, and we prove only for $C \equiv (D \rhd E)$. 

$(\Rightarrow)$: 
Assume $D \rhd E \in \Gamma$. 
Let $\langle \Delta, \sigma \rangle$ be any element of $W$ such that $\langle \Gamma, \tau \rangle R \langle \Delta, \sigma \rangle$ and $\langle \Delta, \sigma \rangle \Vdash D$. 
Then by induction hypothesis, $D \in \Delta$. 
We distinguish the following two cases. 

\begin{itemize}
	\item If $\tau \ast \langle F \rangle \subseteq \sigma$, $\Gamma \prec_F^* \Delta$ and $\Box {\sim}F \in \Delta$ for some $F$, then by Lemma \ref{PL4}, there exists $\Theta \in K_L$ such that $E \in \Theta$, $\Gamma \prec_F^* \Theta$ and ${\sim}F, \Box {\sim}F \in \Theta$. 
Let $\rho : = \tau \ast \langle F \rangle$. 

	\item If not, by Lemma \ref{PL4}, there exists $\Theta \in K_L$ such that $E \in \Theta$ and $\Gamma \prec \Theta$ because $\Gamma \prec_\bot \Delta$. 
Let $\rho : = \tau \ast \langle \bot \rangle$. 
\end{itemize}

	In either case, we have $\rank(\Theta) + 1 \leq \rank(\Gamma)$ and $|\rho| = |\tau| + 1$. Then we obtain 
\[
	\rank(\Theta) + |\rho| = \rank(\Theta) + 1 + |\tau| \leq \rank(\Gamma) + |\tau| \leq \rank(\Gamma_0). 
\]
	It follows $\langle \Theta, \rho \rangle \in W$. 
	By the definition of $S$, we have $\langle \Delta, \sigma \rangle S_{\langle \Gamma, \tau \rangle} \langle \Theta, \rho \rangle$. 
	Also by induction hypothesis, $\langle \Theta, \rho \rangle \Vdash E$. 
	Therefore we conclude $\langle \Gamma, \tau \rangle \Vdash D \rhd E$. 
	
$(\Leftarrow)$: 
Assume $D \rhd E \notin \Gamma$. 
	By Lemma \ref{PL3}, there exists $\Delta \in K_L$ such that $D \in \Delta$, $\Gamma \prec_E^* \Delta$ and $\Box {\sim}E \in \Delta$. 
	Let $\sigma : = \tau \ast \langle E \rangle$, then it is proved that $\langle \Delta, \sigma \rangle$ is an element of $W$ as above. 
	Then $\langle \Delta, \sigma \rangle \Vdash D$ by induction hypothesis. 

	Let $\langle \Theta, \rho \rangle$ be any element of $W$ with $\langle \Delta, \sigma \rangle S_{\langle \Gamma, \tau \rangle} \langle \Theta, \rho \rangle$. 
	Since $\tau \ast \langle E \rangle = \sigma$, $\Gamma \prec_E^* \Delta$ and $\Box {\sim}E \in \Delta$, we have ${\sim}E \in \Theta$ by the definition of $S$. 
	By induction hypothesis, $\langle \Theta, \rho \rangle \nVdash E$. 
	Therefore we conclude $\langle \Gamma, \tau \rangle \nVdash D \rhd E$. 
\end{proof}

Since $\langle \Gamma_0, \epsilon \rangle \in W$ and $A \notin \Gamma_0$, $\langle \Gamma_0, \epsilon \rangle \nVdash A$ by Truth Lemma. 
Therefore $A$ is not valid in the frame of $M$. 
\end{proof}

As a corollary to Theorems \ref{CT1} and \ref{CT2}, we have the decidability of these logics. 

\begin{cor}\label{CC1}
Every logic shown in Figure \ref{Fig1} is decidable. 
\end{cor}

Since every $\IL^-$-frame can be transformed into an $\ILS$-frame, we obtain the following corollary. 

\begin{cor}\label{CC2}
Let $L$ be one of twelve logics in Figure \ref{Fig1} and let $A$ be any formula. 
Then the following are equivalent: 
\begin{enumerate}
	\item $L \vdash A$. 
	\item $A$ is valid in all (finite) $\ILS$-frames in which all axioms of $L$ are valid. 
\end{enumerate}
\end{cor}

\section{Modal incompleteness with respect to $\IL^-$-frames}\label{Sec:Incompl}

In this section, we prove the modal incompleteness of eight logics shown in Figure \ref{Fig2} with respect to $\IL^-$-frames. 
As in Figure \ref{Fig1}, no more line segments can be drawn in the figure.

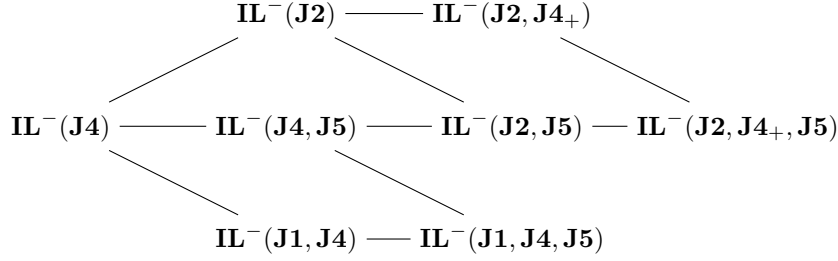
\begin{figure}[ht]
\centering
\begin{tikzpicture}
\node (IL4) at (0,1.5) {$\IL^-(\J{4})$};
\node (IL14) at (3,0) {$\IL^-(\J{1}, \J{4})$};
\node (IL45) at (3,1.5) {$\IL^-(\J{4}, \J{5})$};
\node (IL2) at (3,3){$\IL^-(\J{2})$};
\node (IL145) at (6,0){$\IL^-(\J{1}, \J{4}, \J{5})$};
\node (IL25) at (6,1.5){$\IL^-(\J{2}, \J{5})$};
\node (IL24) at (6,3){$\IL^-(\J{2}, \J{4}_+)$};
\node (IL245) at (9,1.5){$\IL^-(\J{2}, \J{4}_+, \J{5})$};
\draw [-] (IL14)--(IL4);
\draw [-] (IL45)--(IL4);
\draw [-] (IL2)--(IL4);
\draw [-] (IL145)--(IL14);
\draw [-] (IL145)--(IL45);
\draw [-] (IL25)--(IL45);
\draw [-] (IL25)--(IL2);
\draw [-] (IL24)--(IL2);
\draw [-] (IL245)--(IL25);
\draw [-] (IL245)--(IL24);
\end{tikzpicture}
\caption{Sublogics of $\IL$ incomplete with respect to $\IL^-$-frames}
\label{Fig2}
\end{figure}

First, we prove incompleteness of the logics $\IL^-(\J{2})$, $\IL^-(\J{2}, \J{4}_+)$, $\IL^-(\J{2}, \J{5})$ and $\IL^-(\J{2}, \J{4}_+, \J{5})$. 

\begin{prop}\label{ICP1}
$\IL^-(\J{2}, \J{4}_+, \J{5}) \nvdash \J{2}_+$. 
\end{prop}
\begin{proof}
Let $F = \langle W, R, \{S_x\}_{x \in W} \rangle$ be the $\ILS$-frame defined as follows: 

\begin{enumerate}
	\item $W := \{x, y_0, y_1, y_2\}$; 
	\item $R := \{(x, y_0), (x, y_1), (x, y_2)\}$; 
	\item $y_0 S_x V :\iff V \supseteq \{y_1, y_2\}$;\\
	$y_1 S_x V :\iff V \supseteq \{y_2\}$;\\
	$y_2 S_x V :\iff V \supseteq \{y_0, y_1, y_2\}$.
\end{enumerate}

\begin{figure}[ht]
\centering
\begin{tikzpicture}
\node [draw, circle] (x) at (1,0) {$x$};
\node [draw, circle] (y0) at (0,1.5) {$y_0$};
\node [draw, circle] (y1) at (1,1.5) {$y_1$};
\node [draw, circle] (y2) at (2,1.5) {$y_2$};
\draw [thick, ->] (x)--(y0);
\draw [thick, ->] (x)--(y1);
\draw [thick, ->] (x)--(y2);
\end{tikzpicture}
\label{Fig3}
\end{figure}

By Monotonicity of $S_x$, $F$ is actually an $\ILS$-frame. 
First, we prove that $\J{2}$, $\J{4}_+$ and $\J{5}$ are valid in $F$. 

\begin{itemize}
	\item $\J{4}_+$: If $y S_x V$, then $V \cap R[x] = V \setminus \{x\}$. 
	By the definition of $S_x$, we have $y S_x (V \setminus \{x\})$. 
	Thus $y S_x(V \cap R[x])$. 
	By Proposition \ref{4P6}, $\J{4}_+$ is valid in $F$. 

	\item $\J{2}$: Since $\IL^-(\J{4}_+) \vdash \J{4}$, $\J{4}$ is also valid in $F$. 
	Suppose $y S_x V$ and $\forall z \in V \cap R[x] (z S_x U_z)$. 
	Then $y_2 \in V$ if $y$ is either $y_0$, $y_1$ or $y_2$. 
	Also since $y_2 \in V \cap R[x]$, there exists $U_{y_2} \subseteq W$ such that $y_2 S_x U_{y_2}$. 
	By the definition of $S_x$, $U_{y_2} \supseteq \{y_0, y_1, y_2\}$. 
	Thus $\bigcup_{z \in V \cap R[x]} U_z \supseteq \{y_0, y_1, y_2\}$. 
	Then we have $y S_x (\bigcup_{z \in V \cap R[x]} U_z)$ if $y$ is either $y_0$, $y_1$ or $y_2$. 
	Therefore $\J{2}$ is valid in $F$ by Proposition \ref{2P5}. 

	\item $\J{5}$: Since there are no $y, z \in W$ such that $x R y$ and $y R z$, by Proposition \ref{5P2}, $\J{5}$ is trivially valid in $F$. 
\end{itemize}

It suffices to show that $\J{2}_+$ is not valid in $F$. 
Let $V_0 = \{y_1\}$ and $V_1 = \{y_2\}$, then $y_0 S_x (V_0 \cup V_1)$. 
Also let $U_{y_1} = \{y_2\}$, then $\forall z \in V_0 \cap R[x] (z S_x U_z)$. 
On the other hand, since $\bigcup_{z \in V_0 \cap R[x]} U_z \cup V_1 = U_{y_1} \cup V_1 = \{y_2\} \cup \{y_2\} = \{y_2\}$, $y_0 S_x (\bigcup_{z \in V_0 \cap R[x]} U_z \cup V_1)$ does not hold. 
	Therefore $\J{2}_+$ is not valid in $F$ by Proposition \ref{2P6}. 
\end{proof}

\begin{cor}\label{ICC1}
Let $L$ be any logic with $\IL^-(\J{2}) \subseteq L \subseteq \IL^-(\J{2}, \J{4}_+, \J{5})$. 
Then $L$ is not complete with respect to $\IL^-$-frames. 
\end{cor}
\begin{proof}
Let $F$ be any $\IL^-$-frame in which all axioms of $L$ are valid. 
Then $\J{2}$ is valid in $F$, and hence $\J{2}_+$ is also valid in $F$ by Proposition \ref{2P4}. 
However, by Proposition \ref{ICP1}, $L \nvdash \J{2}_+$. 
Therefore $L$ is not complete with respect to $\IL^-$-frames. 
\end{proof}

Secondly, we prove incompleteness of the logics $\IL^-(\J{4})$, $\IL^-(\J{1}, \J{4})$, $\IL^-(\J{4}, \J{5})$ and $\IL^-(\J{1}, \J{4}, \J{5})$. 

\begin{prop}\label{ICP2}
$\IL^-(\J{1}, \J{4}, \J{5}) \nvdash \J{4}_+$. 
\end{prop}
\begin{proof}
We define the $\ILS$-frame $F = \langle W, R, \{S_x\}_{x \in W} \rangle$ as follows: 

\begin{enumerate}
	\item $W := \{x, y_0, y_1, y_2\}$; 
	\item $R := \{(x, y_0), (x, y_1)\}$; 
	\item $y_0 S_x V :\iff V \supseteq \{y_0\}$ or $V \supseteq \{y_1, y_2\}$;\\
	$y_1 S_x V :\iff V \supseteq \{y_1\}$.
\end{enumerate}

\begin{figure}[ht]
\centering
\begin{tikzpicture}
\node [draw, circle] (x) at (0.5,0) {$x$};
\node [draw, circle] (y0) at (0,1.5) {$y_0$};
\node [draw, circle] (y1) at (1,1.5) {$y_1$};
\node [draw, circle] (y2) at (2,1.5) {$y_2$};
\draw [thick, ->] (x)--(y0);
\draw [thick, ->] (x)--(y1);
\end{tikzpicture}
\label{Fig4}
\end{figure}

Indeed, $F$ is an $\ILS$-frame. 
We show $\J{1}$, $\J{4}$ and $\J{5}$ are valid in $F$. 

\begin{itemize}
	\item $\J{1}$: 
	Since $y_0 S_x \{y_0\}$ and $y_1 S_x \{y_1\}$, $\J{1}$ is valid by Proposition \ref{1P3}. 

	\item $\J{4}$: 
	Suppose $y S_x V$. 
	Then whatever $y$ is, either $y_0 \in V$ or $y_1 \in V$. 
	Thus there exists $z \in V$ such that $x R z$. 
	Hence $\J{4}$ is valid in $F$ by Proposition \ref{4P5}. 

	\item $\J{5}$: 
	As in the proof of Proposition \ref{ICP1}, $\J{5}$ is trivially valid in $F$. 
\end{itemize}

Then we show that $\J{4}_+$ is not valid in $F$. 
Let $V = \{y_1, y_2\}$, then $y_0 S_x V$. 
On the other hand, since $V \cap R[x] = \{y_1\}$, $y_0 S_x (V \cap R[x])$ does not hold. 
Therefore $\J{4}_+$ is not valid in $F$ by Proposition \ref{4P6}. 
\end{proof}

\begin{cor}\label{ICC2}
Let $L$ be any logic with $\IL^-(\J{4}) \subseteq L \subseteq \IL^-(\J{1}, \J{4}, \J{5})$. 
Then $L$ is incomplete with respect to $\IL^-$-frames. 
\end{cor}

\section{Modal completeness with respect to $\ILS$-frames}\label{Sec:GCompl}

In this section, we prove eight logics shown in Figure \ref{Fig2} are complete with respect to $\ILS$-frames. 
As in Section \ref{Sec:Compl}, at first we prove the completeness theorem of logics other than $\IL^-(\J{2}, \J{5})$ and $\IL^-(\J{2}, \J{4}_+, \J{5})$. 

\begin{thm}\label{GCT1}
Let $L$ be one of the logics $\IL^-(\J{4})$, $\IL^-(\J{1}, \J{4})$, $\IL^-(\J{4}, \J{5})$, $\IL^-(\J{1}, \J{4}, \J{5})$, $\IL^-(\J{2})$ and $\IL^-(\J{2}, \J{4}_+)$. 
Then for any formula $A$, the following are equivalent: 
\begin{enumerate}
	\item $L \vdash A$. 
	\item $A$ is valid in all (finite) $\ILS$-frames where all axioms of $L$ are valid. 
\end{enumerate}
\end{thm}
\begin{proof}
$(1 \Rightarrow 2)$: Obvious. 

$(2 \Rightarrow 1)$: Assume $L \nvdash A$. 
Let $\Phi$ be any finite adequate set of formulas containing $\{{\sim}A\}$. 
Let $\Gamma_0 \in K_L$ be such that ${\sim}A \in \Gamma_0$. 

We define a model $M = \langle W, R, \{S_x\}_{x \in W}, \Vdash \rangle$ as follows: 
\begin{enumerate}
	\item $W = \{\langle \Gamma, B \rangle : \Gamma \in K_L$ and $B \in \Phi_\rhd\}$;
	\item $\langle \Gamma, B \rangle R \langle \Delta, C \rangle : \iff \Gamma \prec \Delta$;
	\item $\langle \Delta, C \rangle S_{\langle \Gamma, B \rangle} V : \iff$
		\begin{enumerate}
			\item $\langle \Gamma, B \rangle R \langle \Delta, C \rangle$;
			\item For some $\langle \Theta, D \rangle \in V$, $\langle \Gamma, B \rangle R \langle \Theta, D \rangle$; 
			\item The condition $\mathcal{C}_L$ holds. 
		\end{enumerate}
	\item $\langle \Gamma, B \rangle \Vdash p : \iff p \in \Gamma$. 
\end{enumerate}

The condition $\mathcal{C}_L$ depends on $L$ as follows: 

\begin{itemize}
	\item $L \in \{\IL^-(\J{4}), \IL^-(\J{1}, \J{4})\}$: If $\Gamma \prec_C \Delta$, then there exists $\langle \Lambda, G \rangle \in V$ such that ${\sim}C \in \Lambda$. 
	\item $L \in \{\IL^-(\J{4}, \J{5}), \IL^-(\J{1}, \J{4}, \J{5})\}$: If $\Gamma \prec_C \Delta$ and $\Box {\sim}C \in \Delta$, then there exists $\langle \Lambda, G \rangle \in V$ such that ${\sim}C \in \Lambda$. 
	\item $L = \IL^-(\J{2})$: If $\Gamma \prec_C \Delta$, then there exist $\langle \Lambda_0, G \rangle, \langle \Lambda_1, C \rangle \in V$ such that ${\sim}C \in \Lambda_0$ and $\Gamma \prec_C \Lambda_1$. 
	\item $L = \IL^-(\J{2}, \J{4}_+)$: If $\Gamma \prec_C \Delta$, then there exist $\langle \Lambda_0, G \rangle, \langle \Lambda_1, C \rangle \in V$ such that $\Gamma \prec \Lambda_0$, ${\sim}C \in \Lambda_0$ and $\Gamma \prec_C \Lambda_1$. 
\end{itemize}

Since $\bot \in \Phi_{\rhd}$, $\langle \Gamma_{0}, \bot \rangle \in W$. 
Therefore $W$ is non-empty. 
The set $W$ is finite and the relation $R$ is transitive and conversely well-founded. 
Moreover, by Monotonicity of $S_x$, $F = \langle W, R, \{S_x\}_{x \in W} \rangle$ is an $\ILS$-frame. 

\begin{lem}
Every axiom of $L$ is valid in $F$. 
\end{lem}
\begin{proof}
If $\langle \Delta, C \rangle S_{\langle \Gamma, B \rangle} V$, then for some $\langle \Theta, D \rangle \in V$, $\langle \Gamma, B \rangle R \langle \Theta, D \rangle$. 
Thus $\J{4}$ is valid in $F$ by Proposition \ref{4P5}. 

We distinguish the following five cases: 

\begin{itemize}
	\item $L = \IL^-(\J{1}, \J{4})$: 
	Suppose $\langle \Gamma, B \rangle R \langle \Delta, C \rangle$. 
	If $\Gamma \prec_C \Delta$, then ${\sim}C \in \Delta$ because $C \rhd C \in \Gamma$. 
	Hence $\langle \Delta, C \rangle S_{\langle \Gamma, B \rangle} \{\langle \Delta, C \rangle\}$. 
	We conclude that $\J{1}$ is valid in $F$ by Proposition \ref{1P3}. 
	
\item $L = \IL^-(\J{2})$: 
	Assume that $\langle \Delta, C \rangle S_{\langle \Gamma, B \rangle} V$ and for any $ \langle \Delta', C' \rangle \in V \cap R[\langle \Gamma, B \rangle]$, $\langle \Delta', C' \rangle S_{\langle \Gamma, B \rangle} U_{\langle \Delta', C' \rangle}$. 
We distinguish the following two cases: 
\begin{itemize}
	\item If $\Gamma \prec_C \Delta$, then for some $\langle \Lambda_1, C \rangle \in V$, we have $\Gamma \prec_C \Lambda_1$. 
Then $\langle \Lambda_1, C \rangle \in V \cap R[\langle \Gamma, B \rangle]$. 
Since $\langle \Lambda_1, C \rangle S_{\langle \Gamma, B \rangle} U_{\langle \Lambda_1, C \rangle}$ and $\Gamma \prec_C \Lambda_1$, there exist $\langle \Lambda_0', G \rangle, \langle \Lambda_1', C \rangle \in U_{\langle \Lambda_1, C \rangle}$ such that ${\sim}C \in \Lambda_0'$ and $\Gamma \prec_C \Lambda_1'$ by the definition of $S$. 
	Therefore $\langle \Delta, C \rangle S_{\langle \Gamma, B \rangle} U_{\langle \Lambda_1, C \rangle}$. 
	
	\item If $\Gamma \nprec_C \Delta$, then for some $\langle \Theta, D \rangle \in V$, $\langle \Gamma, B \rangle R \langle \Theta, D \rangle$, and hence $\langle \Theta, D \rangle S_{\langle \Gamma, B \rangle} U_{\langle \Theta, D \rangle}$. 
	Then there exists $\langle \Theta', D' \rangle \in U_{\langle \Theta, D \rangle}$ such that $\langle \Gamma, B \rangle R \langle \Theta', D' \rangle$. 
	Therefore, we have $\langle \Delta, C \rangle S_{\langle \Gamma, B \rangle} U_{\langle \Theta, D \rangle}$. 
\end{itemize}
In either case, we obtain $\langle \Delta, C \rangle S_{\langle \Gamma, B \rangle} (\bigcup_{\langle \Delta', C' \rangle \in V \cap R[\langle \Gamma, B \rangle]} U_{\langle \Delta', C' \rangle})$ by Monotonicity. 
Thus we conclude that $\J{2}$ is valid in $F$ by Proposition \ref{2P5}. 

\item $L = \IL^-(\J{2}, \J{4}_+)$: As in the case of $\IL^-(\J{2})$, $\J{2}$ is valid in $F$. 

Suppose $\langle \Delta, C \rangle S_{\langle \Gamma, B \rangle} V$. 
We distinguish the following two cases: 
\begin{itemize}
	\item If $\Gamma \prec_C \Delta$, then there exist $\langle \Lambda_0, G \rangle, \langle \Lambda_1, C \rangle \in V$ such that $\Gamma \prec \Lambda_0$, ${\sim}C \in \Lambda_0$ and $\Gamma \prec_C \Lambda_1$. 
	 Let $V' : = \{\langle \Lambda_0, G \rangle, \langle \Lambda_1, C \rangle\}$. 
	\item If $\Gamma \nprec_C \Delta$, then for some $\langle \Theta, D \rangle \in V$ with $\langle \Gamma, B \rangle R \langle \Theta, D \rangle$, let $V' : = \{\langle \Theta, D \rangle\}$. 
\end{itemize}
In either case, we have $\langle \Delta, C \rangle S_{\langle \Gamma, B \rangle} V'$. 
Also we have $V' \subseteq V \cap R[\langle \Gamma, B \rangle]$. 
Then $\langle \Delta, C \rangle S_{\langle \Gamma, B \rangle} (V \cap R[\langle \Gamma, B \rangle])$ by Monotonicity. 
Therefore $\J{4}_+$ is valid in $F$ by Proposition \ref{4P6}. 

\item $L = \IL^-(\J{4}, \J{5})$: 
Suppose $\langle \Gamma, B \rangle R \langle \Delta, C \rangle$ and $\langle \Delta, C \rangle R \langle \Theta, D \rangle$. 
Let $V: = \{\langle \Theta, D \rangle\}$, then $\langle \Theta, D \rangle \in V \cap R[\langle \Gamma, B \rangle]$. 
If $\Gamma \prec_C \Delta$ and $\Box {\sim}C \in \Delta$, then ${\sim}C \in \Theta$ because $\Delta \prec \Theta$. 
Thus $\langle \Delta, C \rangle S_{\langle \Gamma, B \rangle} V$. 
By Proposition \ref{5P2}, $\J{5}$ is valid in $F$. 

\item $L = \IL^-(\J{1}, \J{4}, \J{5})$: 
As in the cases of $\IL^-(\J{1}, \J{4})$ and $\IL^-(\J{4}, \J{5})$, the axiom schemata $\J{1}$, $\J{4}$ and $\J{5}$ are valid in $F$. 
\end{itemize}
\end{proof}

\begin{lem}[Truth Lemma]
For any $C \in \Phi$ and any $\langle \Gamma, B \rangle \in W$, $C \in \Gamma$ if and only if $\langle \Gamma, B \rangle \Vdash C$. 
\end{lem}
\begin{proof}
We prove by induction on $C$, and we only give a proof of the case $C \equiv (D \rhd E)$. 

$(\Rightarrow)$: 
Assume $D \rhd E \in \Gamma$. 
Let $\langle \Delta, F \rangle$ be any element of $W$ such that $\langle \Gamma, B \rangle R \langle \Delta, F \rangle$ and $\langle \Delta, F \rangle \Vdash D$. 
By induction hypothesis, $D \in \Delta$. 
Since $L$ contains $\J{4}$, by Lemma \ref{PL5}, there exists $\Theta \in K_L$ such that $\Gamma \prec \Theta$ and $E \in \Theta$. 

\begin{itemize}
	\item If $\Gamma \prec_F \Delta$, then by Lemma \ref{PL4}, there exists $\Lambda \in K_L$ such that $E \in \Lambda$ and ${\sim}F \in \Lambda$. 
	In particular, if $L = \IL^-(\J{2}, \J{4}_+)$, then $\Gamma \prec \Lambda$ holds. 
	Moreover, if $L \in \{\IL^-(\J{2}), \IL^-(\J{2}, \J{4}_+)\}$, then we may assume $\Gamma \prec_F \Theta$ by Lemma \ref{PL6}. 
	Let $V : = \{\langle \Theta, F \rangle, \langle \Lambda, F \rangle \}$. 

	\item If $\Gamma \nprec_F \Delta$, then let $V : = \{\langle \Theta, F \rangle\}$. 
\end{itemize}

	In either case, $\langle \Delta, F \rangle S_{\langle \Gamma, B \rangle} V$. 
	By induction hypothesis, $\langle \Theta, F \rangle \Vdash E$ and $\langle \Lambda, F \rangle \Vdash E$. 
	We conclude $\langle \Gamma, B \rangle \Vdash D \rhd E$. 
	
$(\Leftarrow)$: 
Assume $D \rhd E \notin \Gamma$. 
Then by Lemma \ref{PL3}, there exists $\Delta \in K_L$ such that $D \in \Delta$ and $\Gamma \prec_E \Delta$. 
Moreover if $L$ contains $\J{5}$, then $\Box {\sim}E \in \Delta$ also holds. 
We have $\langle \Delta, E \rangle \Vdash D$ by induction hypothesis. 
Let $V$ be any subset of $W$ such that $\langle \Delta, E \rangle S_{\langle \Gamma, B \rangle} V$. 
By the definition of $S$, there exists $\langle \Lambda, G \rangle \in V$ such that ${\sim}E \in \Lambda$. 
Then by induction hypothesis, $\langle \Lambda, G \rangle \nVdash E$. 
Thus we obtain $\langle \Gamma, B \rangle \nVdash D \rhd E$. 
\end{proof}

Since $\langle \Gamma_0, \bot \rangle \in W$ and $A \notin \Gamma_0$, it follows from Truth Lemma that $\langle \Gamma_0, \bot \rangle \nVdash A$. 
Thus $A$ is not valid in the frame of $M$. 
\end{proof}

At last, we prove the completeness of the logics $\IL^-(\J{2}, \J{5})$ and $\IL^-(\J{2}, \J{4}_+, \J{5})$ with respect to $\ILS$-frames. 

\begin{thm}\label{GCT2}
Let $L$ be one of $\IL^-(\J{2}, \J{5})$ and $\IL^-(\J{2}, \J{4}_+, \J{5})$. 
Then for any formula $A$, the following are equivalent: 
\begin{enumerate}
	\item $L \vdash A$. 
	\item $A$ is valid in all (finite) $\ILS$-frames where all axioms of $L$ are valid. 
\end{enumerate}
\end{thm}
\begin{proof}
$(1 \Rightarrow 2)$: Straightforward. 

$(2 \Rightarrow 1)$: 
Suppose $L \nvdash A$. 
Let $\Phi$ be any finite adequate set containing ${\sim}A$. 
Let $\Gamma_0 \in K_L$ be such that ${\sim}A \in \Gamma_0$. 
For each $\Gamma \in K_L$, $\rank(\Gamma)$ is defined as in the proof of Theorem \ref{CT2}. 
Let $k := \max \{\rank(\Gamma) : \Gamma \in K_L\}$. 

We define the model $M : = \langle W, R, \{S_x\}_{x \in W}, \Vdash \rangle$ as follows: 
\begin{enumerate}
	\item $W = \{\langle \Gamma, \tau \rangle : \Gamma \in K_L$ and $\tau$ is a finite sequence of elements of $\Phi_\rhd$ with $\rank(\Gamma) + |\tau| \leq k\}$; 
	\item $\langle \Gamma, \tau \rangle R \langle \Delta, \sigma \rangle : \iff \Gamma \prec \Delta$ and $\tau \subsetneq \sigma$; 
	\item $\langle \Delta, \sigma \rangle S_{\langle \Gamma, \tau \rangle} V : \iff$
		\begin{enumerate}
			\item $\langle \Gamma, \tau \rangle R \langle \Delta, \sigma \rangle$;
			\item For some $\langle \Theta, \rho \rangle \in V$, $\langle \Gamma, \tau \rangle R \langle \Theta, \rho \rangle$; 
			\item If $\tau \ast \langle C \rangle \subseteq \sigma$, $\Gamma \prec_C^* \Delta$ and $\Box {\sim}C \in \Delta$, then the condition $\mathcal{C}_L$ holds. 
		\end{enumerate}
	\item $\langle \Gamma, \tau \rangle \Vdash p : \iff p \in \Gamma$. 
\end{enumerate}
The condition $\mathcal{C}_L$ depends on $L$ as follows: 
\begin{itemize}
	\item $\IL^-(\J{2}, \J{5})$: There exist $\langle \Lambda_1, \rho_1 \rangle, \langle \Lambda_2, \rho_2 \rangle \in V$ such that $\tau \ast \langle C \rangle \subseteq \rho_2$, ${\sim}C \in \Lambda_1$, $\Gamma \prec_C^* \Lambda_2$ and $\Box {\sim}C \in \Lambda_2$. 
	\item $\IL^-(\J{2}, \J{4}_+, \J{5})$: There exist $\langle \Lambda_1, \rho_1 \rangle, \langle \Lambda_2, \rho_2 \rangle \in V$ such that $\tau \ast \langle C \rangle \subseteq \rho_1$, $\tau \ast \langle C \rangle \subseteq \rho_2$, $\Gamma \prec \Lambda_1$, ${\sim}C \in \Lambda_1$, $\Gamma \prec_C^* \Lambda_2$ and $\Box {\sim}C \in \Lambda_2$. 
\end{itemize}

Let $\epsilon$ be the empty sequence. 
Then $\rank(\Gamma_0) + |\epsilon| \leq k$, and hence $\langle \Gamma_0, \epsilon \rangle \in W$. 
Therefore $W$ is a non-empty set. 
Then $\langle W, R, \{S_x\}_{x \in W} \rangle$ is a finite $\ILS$-frame. 

\begin{lem}
Every axiom of $L$ is valid in the frame $F = \langle W, R, \{S_x\}_{x \in W} \rangle$ of $M$. 
\end{lem}
\begin{proof}
$\J{2}$: 
It is easy to show that $\J{4}$ is valid in $F$ (see Proposition \ref{4P5}).  
Suppose $\langle \Delta, \sigma \rangle S_{\langle \Gamma, \tau \rangle} V$ and for any $\langle \Delta', \sigma' \rangle \in V \cap R[\langle \Gamma, \tau \rangle]$, $\langle \Delta', \sigma' \rangle S_{\langle \Gamma, \tau \rangle}U_{\langle \Delta', \sigma' \rangle}$. 
\begin{itemize}
	\item If $\tau \ast \langle C \rangle \subseteq \sigma$, $\Gamma \prec_C^* \Delta$ and $\Box {\sim}C \in \Delta$, then there exists $\langle \Lambda_2, \rho_2 \rangle \in V$ such that $\tau \ast \langle C \rangle \subseteq \rho_2$, $\Gamma \prec_C^* \Lambda_2$ and $\Box {\sim}C \in \Lambda_2$. 
Since $\langle \Lambda_2, \rho_2 \rangle \in V \cap R[\langle \Gamma, \tau \rangle]$, we have $\langle \Lambda_2, \rho_2 \rangle S_{\langle \Gamma, \tau \rangle} U_{\langle \Lambda_2, \rho_2 \rangle}$. 
Since $\tau \ast \langle C \rangle \subseteq \rho_2$, $\Gamma \prec_C^* \Lambda_2$ and $\Box {\sim}C \in \Lambda_2$, by the definition of $S$, the set $U_{\langle \Lambda_2, \rho_2 \rangle}$ satisfies the condition $\mathcal{C}_L$. 
Thus we obtain $\langle \Delta, \sigma \rangle S_{\langle \Gamma, \tau \rangle} U_{\langle \Lambda_2, \rho_2 \rangle}$. 

	\item If not, then let $\langle \Theta, \rho \rangle \in V$ be such that $\langle \Gamma, \tau \rangle R \langle \Theta, \rho \rangle$. 
We have $\langle \Theta, \rho \rangle S_{\langle \Gamma, \tau \rangle} U_{\langle \Theta, \rho \rangle}$ because $\langle \Theta, \rho \rangle \in V \cap R[\langle \Gamma, \tau \rangle]$. 
In particular, $\langle \Delta, \sigma \rangle S_{\langle \Gamma, \tau \rangle} U_{\langle \Theta, \rho \rangle}$. 
\end{itemize}

In either case, by Monotonicity, $\langle \Delta, \sigma \rangle S_{\langle \Gamma, \tau \rangle} (\bigcup_{\langle \Delta', \sigma' \rangle \in V \cap R[\langle \Gamma, \tau \rangle]} U_{\langle \Delta', \sigma' \rangle})$. 
Therefore $\J{2}$ is valid in $F$ by Proposition \ref{2P5}. 

$\J{5}$: 
Suppose $\langle \Gamma, \tau \rangle R \langle \Delta, \sigma \rangle$ and $\langle \Delta, \sigma \rangle R \langle \Theta, \rho \rangle$. 
If there exists $C$ such that $\tau \ast \langle C \rangle \subseteq \sigma$, $\Gamma \prec_C^* \Delta$ and $\Box {\sim}C \in \Delta$, then $\tau \ast \langle C \rangle \subseteq \rho$ because $\sigma \subsetneq \rho$. 
Since $\Gamma \prec_C^* \Delta$ and $\Delta \prec \Theta$, we have $\Gamma \prec_C^* \Theta$ by Lemma \ref{PL2}. 
Also ${\sim}C, \Box {\sim}C \in \Theta$ because $\Delta \prec \Theta$. 
Therefore we obtain $\langle \Delta, \sigma \rangle S_{\langle \Gamma, \tau \rangle} \{\langle \Theta, \rho \rangle\}$. 
By Proposition \ref{5P2}, $\J{5}$ is valid in $F$. 

At last, when $L = \IL^-(\J{2}, \J{4}_+, \J{5})$, we prove that $\J{4}_+$ is valid in $F$. 
Suppose $\langle \Delta, \sigma \rangle S_{\langle \Gamma, \tau \rangle} V$. 
Then there exists $\langle \Theta, \rho \rangle \in V$ such that $\langle \Gamma, \tau \rangle R \langle \Theta, \rho \rangle$, and hence $\langle \Theta, \rho \rangle \in V \cap R[\langle \Gamma, \tau \rangle]$. 
If $\tau \ast \langle C \rangle \subseteq \sigma$, $\Gamma \prec_C^* \Delta$ and $\Box {\sim}C \in \Delta$ for some $C$, then there exist $\langle \Lambda_1, \rho_1 \rangle, \langle \Lambda_2, \rho_2 \rangle \in V$ such that $\tau \ast \langle C \rangle \subseteq \rho_1$, $\tau \ast \langle C \rangle \subseteq \rho_2$, $\Gamma \prec \Lambda_1$, ${\sim}C \in \Lambda_1$, $\Gamma \prec_C^* \Lambda_2$ and $\Box {\sim}C \in \Lambda_2$. 
In particular, $\langle \Lambda_1, \rho_1 \rangle, \langle \Lambda_2, \rho_2 \rangle \in V \cap R[\langle \Gamma, \tau \rangle]$. 
Thus we obtain $\langle \Delta, \sigma \rangle S_{\langle \Gamma, \tau \rangle} (V \cap R[\langle \Gamma, \tau \rangle])$. 
By Proposition \ref{4P6}, $\J{4}_+$ is valid in $F$. 
\end{proof}

\begin{lem}[Truth Lemma]
For any formula $C \in \Phi$ and any $\langle \Gamma, \tau \rangle \in W$, $C \in \Gamma$ if and only if $\langle \Gamma, \tau \rangle \Vdash C$. 
\end{lem}
\begin{proof}
We prove the lemma by induction on the construction of $C$, and we give a proof only for $C \equiv (D \rhd E)$. 

$(\Rightarrow)$: 
Assume $D \rhd E \in \Gamma$. 
Let $\langle \Delta, \sigma \rangle$ be any element of $W$ such that $\langle \Gamma, \tau \rangle R \langle \Delta, \sigma \rangle$ and $\langle \Delta, \sigma \rangle \Vdash D$. 
Then by induction hypothesis, $D \in \Delta$. 
We distinguish the following two cases. 

\begin{itemize}
	\item If $\tau \ast \langle F \rangle \subseteq \sigma$, $\Gamma \prec_F^* \Delta$ and $\Box {\sim}F \in \Delta$, then there exists $\Lambda_1 \in K_L$ such that $E, {\sim}F \in \Lambda_1$ by Lemma \ref{PL4}. 
Moreover, if $L = \IL^-(\J{2}, \J{4}_+, \J{5})$, $\Gamma \prec \Lambda_1$ also holds. 
By Lemma \ref{PL6}, there exists $\Lambda_2 \in K_L$ such that $\Gamma \prec_F^* \Lambda_2$ and $E, \Box {\sim}F \in \Lambda_2$. 
Let $\rho_{1} : = \begin{cases} \epsilon & \ \text{if}\ L = \IL^-(\J{2}, \J{5}), \\ \tau \ast \langle F \rangle & \ \text{if}\ L = \IL^-(\J{2}, \J{4}_{+}, \J{5}) \end{cases}$ and let $\rho_{2} := \tau \ast \langle F \rangle$. 
Then it is easy to see that $\langle \Lambda_1, \rho_1 \rangle, \langle \Lambda_2, \rho_2 \rangle \in W$. 
Let $V : = \{\langle \Lambda_1, \rho_1 \rangle, \langle \Lambda_2, \rho_2 \rangle\}$, then $\langle \Delta, \sigma \rangle S_{\langle \Gamma, \tau \rangle} V$ by the definition of $S$. 
By induction hypothesis, $\langle \Lambda_1, \rho_1 \rangle \Vdash E$ and $\langle \Lambda_2, \rho_2 \rangle \Vdash E$. 
We conclude $\langle \Gamma, \tau \rangle \Vdash D \rhd E$. 

	\item If not, then there exists $\Theta \in K_L$ such that $\Gamma \prec \Theta$ and $E \in \Theta$ by Lemma \ref{PL5}. 
Let $\rho : = \tau \ast \langle E \rangle$, then $\langle \Theta, \rho \rangle \in W$. 
By induction hypothesis, $\langle \Theta, \rho \rangle \Vdash E$. 
Let $V : = \{\langle \Theta, \rho \rangle\}$, then $\langle \Delta, \sigma \rangle S_{\langle \Gamma, \tau \rangle} V$. 
We conclude $\langle \Gamma, \tau \rangle \Vdash D \rhd E$.
\end{itemize}
	
$(\Leftarrow)$: 
Assume $D \rhd E \notin \Gamma$. 
	By Lemma \ref{PL3}, there exists $\Delta \in K_L$ such that $D \in \Delta$, $\Gamma \prec_E^* \Delta$ and $\Box {\sim}E \in \Delta$. 
	Let $\sigma : = \tau \ast \langle E \rangle$, then $\langle \Delta, \sigma \rangle \in W$. 
	We have $\langle \Delta, \sigma \rangle \Vdash D$ by induction hypothesis. 

	Let $V$ be any subset of $W$ with $\langle \Delta, \sigma \rangle S_{\langle \Gamma, \tau \rangle} V$. 
	Since $\tau \ast \langle E \rangle = \sigma$, $\Gamma \prec_E^* \Delta$ and $\Box {\sim}E \in \Delta$, there exists $\langle \Lambda_1, \rho_1 \rangle \in V$ such that ${\sim}E \in \Lambda_1$ by the definition of $S$.
	By induction hypothesis, $\langle \Lambda_1, \rho_1 \rangle \nVdash E$. 
	Therefore we conclude $\langle \Gamma, \tau \rangle \nVdash D \rhd E$. 
\end{proof}

Since $\langle \Gamma_0, \epsilon \rangle \in W$ and $A \notin \Gamma_0$, we obtain $\langle \Gamma_0, \epsilon \rangle \nVdash A$ by Truth Lemma. 
Therefore $A$ is not valid in the frame of $M$. 
\end{proof}

\begin{cor}\label{CCC1}
Every logic shown in Figure \ref{Fig2} is decidable. 
\end{cor}

\section{Concluding Remarks}\label{Sec:CR}

In the previous sections, we investigated the twenty natural sublogics of $\IL$ shown in Figures \ref{Fig1} and \ref{Fig2}. 
We proved that twelve of them are complete with respect to $\IL^-$-frames, but the remaining eight are not.  
Finally, in Section \ref{Sec:GCompl}, we proved that these eight logics are also complete with respect to $\ILS$-frames. 
Consequently, all these twenty logics are also complete with respect to $\ILS$-frames. 
In this situation, one of the referees proposed the following interesting problem. 

\begin{prob}
Does there exist an extension of $\IL^{-}$ incomplete with respect to $\ILS$-frames?
\end{prob}

We introduced these twenty logics to investigate $\IL^-$-frames in detail. 
As a result of our research in the present paper, it can be said that our understanding of the fine structure of $\IL$ has improved in terms of semantical and syntactical aspects. 
Our framework would be useful for finer investigations of some known results of $\IL$ and its extensions.
In addition, investigating whether these newly introduced logics satisfy natural logical properties is an interesting subject in itself.
Along these lines, a research following the present paper is proceeding by the authors (see \cite{IKO}). 

\section*{Acknowledgement}

The authors are grateful to Joost J. Joosten for his helpful comments.  
The author would like to thank the anonymous referees for careful reading and valuable comments. 
The work was supported by JSPS KAKENHI Grant Number JP19K14586.

\bibliographystyle{plain}
\bibliography{ref}

\end{document}